\documentclass[a4paper,11pt]{amsart}

\usepackage{tikz,amsmath,amssymb,amsfonts,amsthm,comment,url,enumitem,color,hyperref}
\usepackage[margin=1in]{geometry}
\usepackage[british]{babel}
\usetikzlibrary{cd}

\theoremstyle{plain}
\newtheorem{prop}{Proposition}[section]
\newtheorem{lem}[prop]{Lemma}
\newtheorem{thm}[prop]{Theorem}
\newtheorem{cor}[prop]{Corollary}
\theoremstyle{definition}
\newtheorem{defn}[prop]{Definition}
\newtheorem*{ack}{Acknowledgements}
\theoremstyle{remark}
\newtheorem*{rmk}{Remark}

\numberwithin{equation}{section}

\makeatletter
\newcommand{\tpmod}[1]{{\@displayfalse\pmod{#1}}}
\newcommand{\where}{{\color{red} WHERE?} \@latex@warning{Reference here!}}
\makeatother
\newcommand{\Stab}{\mathrm{Stab}}
\newcommand{\supp}{\mathrm{supp}}
\newcommand{\Hom}{\mathrm{Hom}}

\title{On equationally Noetherian and residually finite groups}
\author{Motiejus Valiunas}
\address{Instytut Matematyczny, Uniwersytet Wroc{\l}awski, plac Grunwaldzki 2/4, 50-384 Wroc{\l}aw, Poland}
\email{valiunas@math.uni.wroc.pl}
\date{\today}

\begin{document}

\begin{abstract}
The aim of this paper is to compare and contrast the class of residually finite groups with the class of equationally Noetherian groups -- groups over which every system of coefficient-free equations is equivalent to a finite subsystem. It is easy to construct groups that are residually finite but not equationally Noetherian (e.g.\ the direct sum of all finite groups) or vice versa (e.g.\ the additive group $(\mathbb{Q},+)$ of the rationals). However, no explicit examples that are finitely generated seem to appear in the literature.

In this paper, we show that among finitely generated groups, the classes of residually finite and equationally Noetherian groups are similar, but neither of them contains the other. On the one hand, we show that some classes of finitely generated groups which are known to be residually finite, such as abelian-by-polycyclic groups, are also equationally Noetherian (answering a question posed by R.~Bryant in \cite{bryant}). We also give analogous results stating sufficient conditions for a fundamental group of a graph of groups to be equationally Noetherian and to be residually finite. On the other hand, we produce examples of finitely generated non-(equationally Noetherian) groups which are either residually torsion-free nilpotent or conjugacy separable, as well as examples of finitely presented equationally Noetherian groups that are not residually finite.
\end{abstract}
\maketitle
\setcounter{tocdepth}{1}
\tableofcontents

\section{Introduction and main results}

The aim of this paper is to study the relationship between the classes of equationally Noetherian groups (groups in which every system of coefficient-free equations is equivalent to a finite subsystem) and residually finite groups (subgroups of direct products of finite groups); see the next paragraph for precise definitions. Both of these classes of groups have been widely studied separately, although the relationship between them remains fairly unexplored. However, as highlighted below, there are significant parallels between the two classes (especially among finitely generated groups), and many constructions that are known to yield residually finite groups often also give groups that are equationally Noetherian, and vice versa.

To start with, let us fix our terminology. Let $G$ be a group. Given a free group $F_n$ with free basis $\{ X_1,\ldots,X_n \}$ of cardinality $n \geq 1$, an element $s \in F_n * G$ defines a \emph{word map} $s: G^n \to G$ defined by $s(g_1,\ldots,g_n) = \varphi_{g_1,\ldots,g_n}(s)$, where $\varphi_{g_1,\ldots,g_n}: F_n * G \to G$ is the homomorphism sending $X_i$ to $g_i$ and $g$ to $g$ for any $g \in G$. Given any subset $S \subseteq F_n * G$ (called a \emph{system of equations}), the \emph{solution set} of $S$ in $G$ is then defined to be
\[
V_G(S) = \{ (g_1,\ldots,g_n) \in G^n \mid s(g_1,\ldots,g_n) = 1 \text{ for all } s \in S \}.
\]
The group $G$ is said to be \emph{equationally Noetherian} (respectively \emph{strongly equationally Noetherian}) if for every $n \geq 1$ and every $S \subseteq F_n$ (respectively every $S \subseteq F_n * G$), there exists a finite subset $S_0 \subseteq S$ such that $V_G(S_0) = V_G(S)$. Recall also that if $\mathcal{F}$ is a family (or a property) of groups, then $G$ is said to be \emph{residually $\mathcal{F}$} if for every non-trivial element $g \in G$, there exists a group homomorphism $\varphi: G \to Q$ with $Q \in \mathcal{F}$ such that $\varphi(g) \neq 1$. We will mostly be interested in the case when $\mathcal{F}$ is the family of all finite groups; in this case, the notion of a residually $\mathcal{F}$ group coincides with the classical notion of a residually finite group.

\begin{rmk}
Our terminology follows \cite{gh} and is slightly non-standard, in that groups we call `strongly equationally Noetherian' are often called `equationally Noetherian' in the literature (see e.g.\ \cite{bmr}). However, although there are equationally Noetherian groups that are not strongly equationally Noetherian \cite[\S 2.2, Proposition 4]{bmr}, we will mostly be interested in finitely generated groups, and for those the two notions coincide \cite[\S 2.2, Proposition 3]{bmr}.
\end{rmk}

Examples of groups that are residually finite but not equationally Noetherian, or vice versa, are not hard to come by. For instance, any abelian group is equationally Noetherian \cite[\S 2.2, Theorem 1]{bmr}, although there are abelian groups which are not residually finite, such as the quasi-cyclic groups $\mathbb{Z}[\frac1p]/\mathbb{Z}$ or the additive group $(\mathbb{Q},+)$ of the rationals. In the other direction, a direct sum of finite groups is clearly residually finite, but the group $\bigoplus_{n=1}^\infty S_3 \wr C_n$ is not equationally Noetherian \cite[Example 3]{bmro}. However, all of these examples are not finitely generated, and constructing `easy' examples of such finitely generated groups is significantly more complicated.

One of the main goals of this paper is therefore to construct finitely generated groups which are residually finite but not equationally Noetherian, or vice versa. We do this in Proposition~\ref{p:RF=/>EN} and Theorem~\ref{t:EN=/>RF} below. However, before doing this, let us first discuss why such a construction -- and in general the existence of such finitely generated groups -- is not obvious.

\subsection{Equationally Noetherian families}

Let $\mathcal{G}$ be a family of groups. We say that $\mathcal{G}$ is an \emph{equationally Noetherian family} if for all $n \geq 1$ and all $S \subseteq F_n$, there exists a finite subset $S_0 \subseteq S$ such that for all $G \in \mathcal{G}$ we have $V_G(S_0) = V_G(S)$. It is clear that if $\mathcal{G}$ is an equationally Noetherian family, then each group $G \in \mathcal{G}$ is equationally Noetherian. The converse, however, is not true: as D.~Groves and M.~Hull showed in \cite[Example 3.15]{gh}, the family of all finite groups is not equationally Noetherian.

Our concern with equationally Noetherian families stems from the following observation, which is implicit in the work of G.~Baumslag, A.~Myasnikov and V.~Remeslennikov \cite{bmr}.
\begin{lem}[cf {\cite[Theorem B2(3)]{bmr}}] \label{l:RFEN}
Let $\mathcal{F}$ be an equationally Noetherian family, and let $\mathcal{G}$ be a family of residually $\mathcal{F}$ groups. Then $\mathcal{G}$ is an equationally Noetherian family. In particular, every element of $\mathcal{G}$ is equationally Noetherian.
\end{lem}

We give a proof of Lemma~\ref{l:RFEN} in Section \ref{s:enfam}. We use it to study the property of being equationally Noetherian for classes of groups that satisfy various separability properties. In particular, we consider the following:
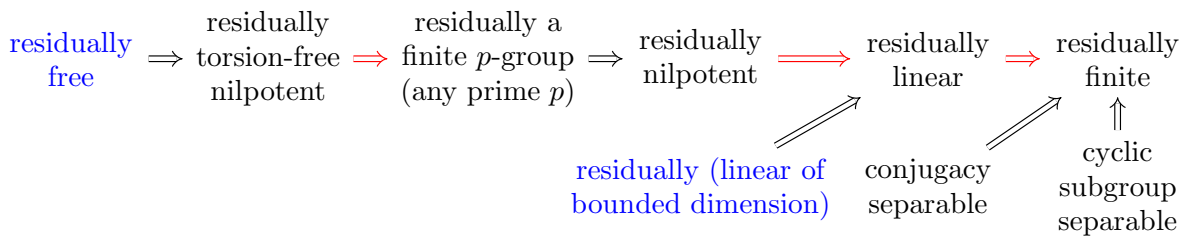
\[
\begin{tikzcd}[column sep=small,row sep=tiny]
\parbox{1.7cm}{\centering\color{blue} residually free} \ar[r,Rightarrow] & \parbox{1.9cm}{\centering residually torsion-free nilpotent} \ar[r,Rightarrow,red] & \parbox{2.3cm}{\centering residually a finite $p$-group (any prime $p$)} \ar[r,Rightarrow] &[-0.9cm] \parbox{1.7cm}{\centering residually nilpotent} \ar[r,Rightarrow,red] &[-0.4cm] \parbox{1.7cm}{\centering residually linear} \ar[r,Rightarrow,red] & \parbox{1.7cm}{\centering residually finite} \\
&&& \parbox{3.5cm}{\centering\color{blue} residually (linear of bounded dimension)} \ar[ru,Rightarrow] & \parbox{1.7cm}{\centering conjugacy separable} \ar[ru,Rightarrow] & \parbox{1.7cm}{\centering cyclic subgroup separable} \ar[u,Rightarrow]
\end{tikzcd}
\]
(see Section~\ref{ss:intro-counterex} for definitions of conjugacy separability and cyclic subgroup separability); here red arrows show implications among finitely generated groups, whereas black arrows denote implications for any group. Indeed, (residually) free groups are residually torsion-free nilpotent \cite{magnus}, (residually) finitely generated torsion-free nilpotent groups are residually finite $p$-groups for any prime $p$ \cite{gruenberg}, and `finitely generated nilpotent' implies \cite{jennings} `finitely generated linear' implies \cite{malcevRF} `residually finite'; the other implications are trivial or easy.

In regard to this picture, we show that only groups shown in blue can be proved to be equationally Noetherian, whereas the other classes contain finitely generated non-(equationally Noetherian) examples. Our discussion can be summarised by the following result, which we present here as a corollary to Lemma~\ref{l:RFEN}. We note that most of the statements listed in Corollary~\ref{c:ENfam} are known to experts although do not seem appear in the literature in the precise form stated here; therefore, we prove them in the present paper.

\begin{cor} \label{c:ENfam}
The following families of groups are equationally Noetherian:
\begin{enumerate}[label={\textup{(\roman*)}}]
\item \label{i:cENfam-free} all free groups;
\item \label{i:cENfam-nilp} all $r$-step nilpotent groups (for any fixed integer $r \geq 1$);
\item \label{i:cENfam-lin} subgroups of $GL_r(R)$ for all commutative unital rings $R$ (for any fixed integer $r \geq 1$).
\end{enumerate}
In particular, any residually free group and any group which is residually (linear of bounded dimension) is equationally Noetherian.

The following families of groups are \emph{not} equationally Noetherian:
\begin{enumerate}[label={\textup{(\roman*)}},resume]
\item \label{i:cENfam-tfn} all finitely generated (torsion-free) nilpotent groups;
\item \label{i:cENfam-p} all finite $p$-groups (for any fixed prime $p$).
\end{enumerate}
\end{cor}
We prove parts \ref{i:cENfam-free}, \ref{i:cENfam-nilp} and \ref{i:cENfam-lin} of Corollary~\ref{c:ENfam} in Section~\ref{s:enfam}. The `in particular' part of the Corollary then follows from parts \ref{i:cENfam-free}, \ref{i:cENfam-lin} and Lemma~\ref{l:RFEN}, whereas parts \ref{i:cENfam-tfn} and \ref{i:cENfam-p} follow from Lemma~\ref{l:RFEN} and Proposition~\ref{p:RF=/>EN}\ref{i:pRF-ut}.

Part \ref{i:cENfam-free} of Corollary~\ref{c:ENfam} is widely known and only included here for the sake of completeness.

Part \ref{i:cENfam-nilp} highlights the difference between equationally Noetherian and strongly equationally Noetherian groups: for instance, if $r \geq 2$ and $G$ is a direct sum of infinitely many non-abelian $r$-step nilpotent groups (and so $G$ is $r$-step nilpotent), then $G$ is not strongly equationally Noetherian \cite[\S 2.2, Proposition 4(1)]{bmr}, but it is equationally Noetherian.

A similar distinction is made by part \ref{i:cENfam-lin}: this should be compared to \cite[Theorem B1]{bmr}, which says that if $R$ is a commutative unital \emph{Noetherian} ring then $GL_r(R)$ is \emph{strongly} equationally Noetherian. The assumption on $R$ being Noetherian is necessary: in Lemma~\ref{l:nnsen} we give an example of a linear group that is not strongly equationally Noetherian.

Part \ref{i:cENfam-tfn} explores the limitations of parts \ref{i:cENfam-nilp} and \ref{i:cENfam-lin}. In particular, it shows that the bound on $r$ cannot be dropped in \ref{i:cENfam-nilp}; as any finitely generated torsion-free nilpotent group is a subgroup of $GL_r(\mathbb{Z})$ for some $r$ \cite{jennings}, this also shows that the bound on $r$ cannot be dropped in \ref{i:cENfam-lin}.

Finally, part \ref{i:cENfam-p} can be seen as a strengthening of \cite[Example 3.15]{gh}, which says that the family of all finite groups is not equationally Noetherian.

\subsection{Evidence for similarity} Below, we summarise known results on equationally Noetherian groups, and discuss analogous results in the setting of residually finite groups.

\begin{description}

\item[Basic constructions] It follows from the definitions that finite groups are both equationally Noetherian and residually finite, and that both of these classes are preserved under taking subgroups and direct products with finitely many terms.

\item[Abelian and linear groups] Any abelian group $G$ and any subgroup $G \leq GL_n(F)$ (for a field $F$) is strongly equationally Noetherian \cite[\S 2.2, Theorems 1 and B1]{bmr}; compare also with Corollary~\ref{c:ENfam}, parts \ref{i:cENfam-nilp} and \ref{i:cENfam-lin}. If such a group $G$ is finitely generated, then it is also residually finite: if $G$ is abelian then this follows from the Structure theorem for finitely generated abelian groups, whereas if $G$ is linear then this follows by a result of A.~I.~Mal'cev \cite{malcevRF}.

\item[Hopficity] It is well-known that all finitely generated residually finite groups $G$ are Hopfian, i.e.\ any surjective homomorphism $\varphi: G \to G$ is an isomorphism, and the same is true for finitely generated equationally Noetherian groups \cite[Corollary 3.14]{gh}. If the `finitely generated' assumption is dropped, then counterexamples to this claim exist: for instance, the infinitely generated free group $F_\infty$ is non-Hopfian but residually finite and equationally Noetherian (as it is a subgroup of the finitely generated linear group $F_2$).

\item[Finite extensions] It is easy to see that a virtually residually finite group is residually finite. The fact that a virtually (strongly) equationally Noetherian group is (strongly) equationally Noetherian is due to G.~Baumslag, A.~Myasnikov and V.~Roman'kov \cite[Theorem 1]{bmro}.

\item[Wreath products] By a result of K.~W.~Gruenberg \cite[Theorem 3.2]{gruenberg}, a regular wreath product $G \wr H$ is residually finite if and only if both $G$ and $H$ are residually finite and either $G$ is abelian or $H$ is finite. If $H$ is finite and $G$ is (strongly) equationally Noetherian, then $G^H$ is also (strongly) equationally Noetherian and hence so is its finite extension $G \wr H$. If $G$ is non-abelian and $H$ is infinite, then $G \wr H$ is not strongly equationally Noetherian \cite[Proposition 1]{bmro}, and it can be seen to be not equationally Noetherian if $H$ is not locally finite. The case when $G$ is abelian and $H$ is infinite is somewhat more complicated: see Proposition~\ref{p:RF=/>EN}\ref{i:pRF-wp} below.

\item[Abelian-by-polycyclic groups] By a result of P.~Hall \cite[Theorem 1]{hallAbNRF}, all finitely generated abelian-by-nilpotent groups are residually finite, and this has been generalised to finitely generated abelian-by-polycyclic groups by A.~V.~Jategaonkar \cite[Theorem 3]{jategaonkar}. In \cite[Corollary 3.7]{bryant}, R.~M.~Bryant proved that finitely generated abelian-by-nilpotent groups are equationally Noetherian, and Theorem~\ref{t:abpolyc} below generalises this to abelian-by-polycyclic groups.

\item[Rigid soluble groups] A group $G$ is said to be \emph{rigid} if it possesses a chain $G = G_0 \geq G_1 \geq \cdots \geq G_r = \{1\}$ of normal subgroups such that each quotient $G_i/G_{i+1}$ is abelian and torsion-free as a $\mathbb{Z}[G/G_i]$-module. Such groups are equationally Noetherian by a result of N.~S.~Romanovskii \cite[Theorem 1]{roman}. It is also known \cite[Proposition 5.9]{mr} that any finitely generated rigid group embeds in an iterated wreath product $\mathbb{Z}^m \wr (\mathbb{Z}^m \wr (\cdots (\mathbb{Z}^m \wr \mathbb{Z}^m) \cdots ))$, and so finitely generated rigid groups are residually finite by the wreath product case discussed above.

\item[Free products] If $G$ and $H$ are residually finite groups, it is well-known that $G * H$ is residually finite. A significantly more difficult result, due to Z.~Sela \cite[Theorem 9.1]{sela}, shows that if $G$ and $H$ are equationally Noetherian then so is $G * H$.

\item[Graph products] Let $\Gamma$ be a finite simple graph and let $\mathcal{G} = \{ G_v \mid v \in V(\Gamma) \}$ be a collection of groups. If each $G_v$ is residually finite, then, by a result of E.~Green \cite[Corollary 5.4]{green}, so is the graph product $\Gamma\mathcal{G}$. If the graph $\Gamma$ has girth at least $6$ and if each $G_v$ is equationally Noetherian, then a result of the author \cite[Theorem E]{me4} tells that $\Gamma\mathcal{G}$ is equationally Noetherian.

\item[$\mathcal{G}$-limit groups] Even though the family of all finite groups is not equationally Noetherian \cite[Example 3.15]{gh}, it contains large equationally Noetherian subfamilies: see Corollary~\ref{c:ENfam}\ref{i:cENfam-lin}, for instance. Given such a subfamily $\mathcal{G}$, one may consider $\mathcal{G}$-limit groups; see Section~\ref{s:limit} for a precise definition. These (finitely generated) groups will be both equationally Noetherian and residually finite (by Proposition~\ref{p:limit} below), and cover many of the examples given above.

\item[Hyperbolic and relatively hyperbolic groups] A result of Z.~Sela \cite[Theorem 1.22]{selahyp} says that a torsion-free word hyperbolic group $G$ is equationally Noetherian. This result has been generalised to any hyperbolic group $G$ (C.~Reinfeldt and R.~Weidmann \cite[Corollary 6.13]{rw}), and more recently to any group $G$ that is hyperbolic relative to a collection of equationally Noetherian subgroups (D.~Groves and M.~Hull \cite[Theorem D]{gh}).

A very well-known open question, asked by M.~Gromov \cite{gromov87}, is whether all hyperbolic groups are residually finite. By a result of D.~Osin \cite[Corollary 1.8]{osin07}, this is equivalent to asking whether any group hyperbolic relative to a collection of residually finite subgroups is residually finite.

\end{description}

As mentioned above, we generalise a result of R.~M.~Bryant \cite[Corollary 3.7]{bryant} showing that finitely generated abelian-by-nilpotent groups are equationally Noetherian. This also answers a question posed by Bryant in \cite{bryant}.
\begin{thm} \label{t:abpolyc}
Any abelian-by-polycyclic group is equationally Noetherian.
\end{thm}

For $\mathcal{G}$-limit groups, we prove the following.
\begin{prop} \label{p:limit}
Let $\mathcal{G}$ be an equationally Noetherian family of groups, and let $L$ be a $\mathcal{G}$-limit group. Then $L$ is equationally Noetherian. Moreover, if all the groups in $\mathcal{G}$ are finite, then $L$ is residually finite.
\end{prop}

We prove Theorem~\ref{t:abpolyc} in Section~\ref{s:abpolyc}, and Proposition~\ref{p:limit} in Section~\ref{s:limit}.

\begin{rmk}
The group $L$ in Proposition~\ref{p:limit} would not need be residually finite if we dropped the assumption on $\mathcal{G}$ being equationally Noetherian. For example, if $G$ is a non-abelian finite group then the group $L = G \wr \mathbb{Z}$ is a $\mathcal{G}$-limit group for the family $\mathcal{G} = \{ G \wr C_n \mid n \geq 1 \}$ of finite groups, but $L$ is not residually finite \cite[Theorem 3.2]{gruenberg}. Indeed, it is even possible that a $\mathcal{G}$-limit group (where $\mathcal{G}$ is a family of finite groups) is non-Hopfian: see \cite[Proof of Example 3.15]{gh}.
\end{rmk}

Given the results above, one might be tempted to conjecture that, among finitely generated groups, the classes of equationally Noetherian and residually finite groups coincide. However, the following arguments show that this is not the case.

\subsection{The counterexamples}
\label{ss:intro-counterex}

We now demonstrate how one can construct groups that are residually finite but not equationally Noetherian, or vice versa. To construct a residually finite group that is not equationally Noetherian, we use the following result.

\begin{thm}[{J.~S.~Wilson \cite[Theorem A]{wilson}}] \label{t:wilsonRF}
Every countable residually finite group $G$ embeds in a $2$-generator residually finite group $\widehat{G}$, which may be chosen so that if $G$ is soluble, residually soluble or residually nilpotent, then so is $\widehat{G}$.
\end{thm}

Thus, if $G$ is a countable soluble residually finite group that is not equationally Noetherian -- for instance, we could take $\bigoplus_{n=1}^\infty S_3 \wr C_n$ as discussed above -- then, by Theorem~\ref{t:wilsonRF}, $G$ is a subgroup of a $2$-generator soluble residually finite group $\widehat{G}$. As equational Noetherianity is preserved under taking subgroups, it follows that $\widehat{G}$ is not equationally Noetherian.

Here we construct finitely generated groups that are not equationally Noetherian and possess separability properties that are stronger than residual finiteness. Recall that a subset $A \subseteq G$ of a group $G$ is said to be \emph{separable} in $G$ if for any $g \in G \setminus A$, there exists a group homomorphism $\varphi: G \to Q$ such that $Q$ is finite and $\varphi(g) \notin \varphi(A)$. A group $G$ is then said to be \emph{conjugacy separable} (respectively \emph{cyclic subgroup separable}) if any conjugacy class (respectively any cyclic subgroup) is separable in $G$. As a group $G$ is residually finite precisely when $\{1\}$ is separable in $G$, it follows that all conjugacy separable groups and all cyclic subgroup separable groups are residually finite.

\begin{prop} \label{p:RF=/>EN} ~
\begin{enumerate}[label={\textup{(\roman*)}}]
\item \label{i:pRF-ut} There exists a finitely generated subgroup $G_1$ of $\prod_{r=1}^\infty UT_r(\mathbb{Z})$ that is not equationally Noetherian, where $UT_r(\mathbb{Z}) \leq GL_r(\mathbb{Z})$ is the subgroup of upper unitriangular matrices. In particular, $G_1$ is residually torsion-free nilpotent and so residually a finite $p$-group for any fixed prime $p$.
\item \label{i:pRF-wp} The group $G_2 = G \wr H$, where $G$ is a non-trivial finitely generated abelian group and $H$ is a finitely generated group that has an infinite locally finite subgroup, is not equationally Noetherian. In particular, if $H$ is cyclic subgroup separable and conjugacy separable, then so is $G_2$.
\end{enumerate}
\end{prop}
An example of a group $G_2$ satisfying the assumptions of Proposition~\ref{p:RF=/>EN}\ref{i:pRF-wp} is $C_2 \wr (C_2 \wr \mathbb{Z})$. We prove part \ref{i:pRF-ut} of Proposition~\ref{p:RF=/>EN} in Section \ref{ss:RTFNnotEN} and part \ref{i:pRF-wp} in Section \ref{ss:CSnotEN}. For part \ref{i:pRF-wp}, we find necessary and sufficient conditions for a regular wreath product of two groups to be cyclic subgroup separable, as follows.

\begin{thm} \label{t:css}
Let $G$ and $H$ be groups. Then $G \wr H$ is cyclic subgroup separable if and only if both $G$ and $H$ are cyclic subgroup separable and either $G$ is abelian or $H$ is finite.
\end{thm}
We prove Theorem~\ref{t:css} in Section \ref{s:cssep}.

In the other direction, we prove the following result.
\begin{thm} \label{t:EN=/>RF}
Let $G$ be a connected Zariski closed Lie subgroup of $GL_m(\mathbb{R})$, and let $p: H \to G$ be a covering map such that $H$ is a connected Lie group and $p$ is a group homomorphism. Then $H$ is strongly equationally Noetherian. In particular, for $n \geq 2$, the preimage of $Sp_{2n}(\mathbb{Z})$ under the universal covering map $\widetilde{Sp_{2n}(\mathbb{R})} \to Sp_{2n}(\mathbb{R})$ is a finitely presented equationally Noetherian group that is not residually finite.
\end{thm}
We prove Theorem~\ref{t:EN=/>RF} in Section \ref{ss:ENnotRF}.

\subsection{Graphs of groups}

Finally, we devise a collection of sufficient conditions for the fundamental group $\pi_1(\mathcal{G})$ of a graph of groups $\mathcal{G}$ to be equationally Noetherian, and show that analogous conditions are sufficient for $\pi_1(\mathcal{G})$ to be residually finite. Briefly, a (finite) graph of groups $\mathcal{G}$ consists of a (finite) graph with vertex set $V(\mathcal{G})$ and edge set $E(\mathcal{G})$ (each edge viewed as a pair of directed edges), collections $\{ G_v \mid v \in V(\mathcal{G}) \}$ of vertex groups and $\{ G_e \mid e \in E(\mathcal{G}) \}$ of edge groups, and injective homomorphisms $\iota_e: G_e \to G_{i(e)}$ (where $i(e)$ is the initial vertex of $e \in E(\mathcal{G})$); see Definition \ref{d:gg} for a precise definition of a graph of groups $\mathcal{G}$, its fundamental group $\pi_1(\mathcal{G})$ and the corresponding Bass--Serre tree $\mathcal{T_G}$.

In the following Theorem, we say a subset $A \subseteq G$ of a group $G$ is \emph{quasi-algebraic} if for all $n \geq 1$ and all $S \subseteq F_n$, there exists a finite subset $S_0 \subseteq S$ such that $V_{G,A}(S_0) = V_{G,A}(S)$, where for any $T \subseteq F_n$ we define $V_{G,A}(T) = \{ (g_1,\ldots,g_n) \in G^n \mid t(g_1,\ldots,g_n) \in A \text{ for all } t \in T \}$. Note that a normal subgroup $N \unlhd G$ is quasi-algebraic if and only if $G/N$ is equationally Noetherian, and in particular $\{1\}$ is quasi-algebraic in $G$ if and only if $G$ is equationally Noetherian. Examples of quasi-algebraic subsets in a group $G$ include cosets of finite-index subgroups and, if $G$ is equationally Noetherian, solution sets of equations; we refer the reader to Section~\ref{ss:qalg} for a more detailed discussion. See also Section~\ref{ss:acyl} for the definition of an acylindrical action and a sufficient condition for the action of $\pi_1(\mathcal{G})$ on $\mathcal{T_G}$ to be acylindrical (Lemma~\ref{l:acyl}). 

\begin{thm} \label{t:comb}
Let $\mathcal{G}$ be a finite graph of groups. Suppose that
\begin{enumerate}[label={\textup{(\roman*)}}]
\item \label{i:tcomb-ext} for each $e \in E(\mathcal{G})$, the isomorphism $\phi_e = \iota_{\overline{e}} \circ \iota_e^{-1}: \iota_e(G_e) \to \iota_{\overline{e}}(G_e)$ extends to a homomorphism $\overline\phi_e: G_{i(e)} \to G_{i(\overline{e})}$; and
\item \label{i:tcomb-fin1v} for each $v \in V(\mathcal{G})$, there are only finitely many endomorphisms of the form $\overline\phi_{e_n} \circ \cdots \circ \overline\phi_{e_1}: G_v \to G_v$, where $e_1 \cdots e_n$ ranges over all closed paths in $\mathcal{G}$ starting and ending at $v$.
\end{enumerate}
Then the following hold.
\begin{enumerate}[label=\textup{(\alph*)}]
\item \label{i:tcomb-RF} The group $\pi_1(\mathcal{G})$ is residually finite if and only if $G_v$ is residually finite for all $v \in V(\mathcal{G})$ and $\iota_e(G_e)$ is separable in $G_{i(e)}$ for all $e \in E(\mathcal{G})$.
\item \label{i:tcomb-EN} Suppose that the action of $\pi_1(\mathcal{G})$ on $\mathcal{T_G}$ is acylindrical. Then the group $\pi_1(\mathcal{G})$ is equationally Noetherian if and only if $G_v$ is equationally Noetherian for all $v \in V(\mathcal{G})$ and $\iota_e(G_e)$ is quasi-algebraic in $G_{i(e)}$ for all $e \in E(\mathcal{G})$.
\end{enumerate}
\end{thm}

We prove Theorem~\ref{t:comb} in Section~\ref{s:comb}. In both parts \ref{i:tcomb-RF} and \ref{i:tcomb-EN}, the implication ($\Rightarrow$) actually works in a more general setting: our proof of this implication does not use the assumption \ref{i:tcomb-fin1v}, nor the assumption of acylindricity in part \ref{i:tcomb-EN}.


The main new contribution of Theorem~\ref{t:comb} is part \ref{i:tcomb-EN}. In contrast, even though part \ref{i:tcomb-RF} does not seem to be explicit in the literature, many special cases or related results have been shown: see, for instance, \cite{be,bt,hempel,arv}. We state and prove part \ref{i:tcomb-RF} here to complement part \ref{i:tcomb-EN} -- that is, to highlight the similarity between the classes of equationally Noetherian and residually finite groups.

We provide a couple of corollaries that can be easily deduced from Theorem~\ref{t:comb}.

\begin{cor} \label{c:graphs-same}
Let $\mathcal{G}$ be a finite graph of groups. Suppose that there exists a group $G$ and isomorphisms $\psi_v: G_v \to G$ for all $v \in V(\mathcal{G})$ such that $\psi_{i(e)} \circ \iota_e = \psi_{i(\overline{e})} \circ \iota_{\overline{e}}$ for all $e \in E(\mathcal{G})$. Then the statements \ref{i:tcomb-RF} and \ref{i:tcomb-EN} in Theorem~\ref{t:comb} hold.
\end{cor}

Informally, Corollary~\ref{c:graphs-same} gives necessary and sufficient conditions for $\pi_1(\mathcal{G})$ to be residually finite and (when the action $\pi_1(\mathcal{G}) \curvearrowright \mathcal{T_G}$ is acylindrical) equationally Noetherian, in the case when all the vertex groups $G_v$ are `the same' and all the edge maps $\iota_e$ are `inclusions'.

\begin{cor} \label{c:graphs-HNN}
Let $G$ be a group, $\varphi: G \to G$ an automorphism of finite order, and $H \leq G$. Consider the HNN-extension $K = \langle G,t \mid h^t = \varphi(h) \text{ for all } h \in H \rangle$. 
\begin{enumerate}[label=\textup{(\alph*)}]
\item \label{i:combHNN-RF} $K$ is residually finite if and only if $G$ is residually finite and $H$ is separable in $G$.
\item \label{i:combHNN-EN} Suppose that the action of $K$ on the corresponding Bass--Serre tree is acylindrical. Then $K$ is equationally Noetherian if and only if $G$ is equationally Noetherian and $H$ is quasi-algebraic in $G$.
\end{enumerate}
\end{cor}

\begin{rmk}
Corollary~\ref{c:graphs-HNN}\ref{i:combHNN-RF} may be compared with \cite[Theorem~A]{logan}, which proves the same statement without the assumption that $\varphi$ has finite order but with an additional assumption that $G$ is finitely generated. Both of these assumptions cannot be dropped simultaneously: indeed, if $G$ is a non-abelian finite group, then $\bigoplus_{j \in \mathbb{Z}} G$ is residually finite, but the wreath product $K = G \wr \mathbb{Z} \cong \left( \bigoplus_{j \in \mathbb{Z}} G \right) \rtimes \mathbb{Z}$ is not \cite[Theorem~3.2]{gruenberg}.
\end{rmk}

\begin{ack}
The author would like to thank Armando Martino and Ashot Minasyan for valuable discussions, as well as the referees for their insightful comments.
\end{ack}

\section{Equationally Noetherian families} \label{s:enfam}

In this section, we prove Lemma~\ref{l:RFEN} and parts \ref{i:cENfam-free}, \ref{i:cENfam-nilp} and \ref{i:cENfam-lin} of Corollary~\ref{c:ENfam}. We first give (fairly straightforward) proofs of Lemma~\ref{l:RFEN} and Corollary \ref{c:ENfam}\ref{i:cENfam-free}.

\begin{proof}[Proof of Lemma~\ref{l:RFEN}]
Let $S \subseteq F_n$ be a system of equations. As $\mathcal{F}$ is equationally Noetherian, there exists a finite subset $S_0 \subseteq S$ such that $V_F(S) = V_F(S_0)$ for all $F \in \mathcal{F}$. We claim that we also have $V_G(S) = V_G(S_0)$ for all $G \in \mathcal{G}$.

Thus, let $G \in \mathcal{G}$. It is clear that $V_G(S) \subseteq V_G(S_0)$. Conversely, let $\mathbf{g} = (g_1,\ldots,g_n) \in G^n$ be such that $\mathbf{g} \notin V_G(S)$. Then $s(g_1,\ldots,g_n) \neq 1$ for some $s \in S$. As $G$ is residually $\mathcal{F}$, it follows that there exist $F \in \mathcal{F}$ and a homomorphism $\pi: G \to F$ such that
\[
s(\pi(g_1),\ldots,\pi(g_n)) = \pi(s(g_1,\ldots,g_n)) \neq 1.
\]
In particular, $(\pi(g_1),\ldots,\pi(g_n)) \notin V_F(S) = V_F(S_0)$, and so there exists $s_0 \in S_0$ such that
\[
\pi(s_0(g_1,\ldots,g_n)) = s_0(\pi(g_1),\ldots,\pi(g_n)) \neq 1.
\]
Therefore, $s_0(g_1,\ldots,g_n) \neq 1$ and so $\mathbf{g} \notin V_G(S_0)$. This shows that $V_G(S) \supseteq V_G(S_0)$ and so $V_G(S) = V_G(S_0)$, as claimed.
\end{proof}

\begin{proof}[Proof of Corollary~\ref{c:ENfam}\ref{i:cENfam-free}]
By \cite[Lemma 3.11]{gh}, a family $\mathcal{G}$ of groups is equationally Noetherian if and only if the family of finitely generated subgroups of groups in $\mathcal{G}$ is equationally Noetherian. Thus, as subgroups of free groups are free, it is enough to show that the family $\mathcal{G}$ of \emph{finitely generated} free groups are equationally Noetherian. But $\mathcal{G}$ is precisely the family of finitely generated subgroups of $F_2$, and so it is enough to show that the family $\{F_2\}$ is equationally Noetherian, that is, the group $F_2$ is equationally Noetherian. But this follows, for instance, since $F_2$ is linear over $\mathbb{Z}$.
\end{proof}

The proofs of parts \ref{i:cENfam-nilp} and \ref{i:cENfam-lin} of Corollary~\ref{c:ENfam} are somewhat more involved. For part \ref{i:cENfam-nilp}, we use the existence of free nilpotent groups, as well as the fact that finitely generated nilpotent groups are \emph{Noetherian}: that is, all subgroups of a finitely generated nilpotent group are finitely generated.

Recall that a group $G$ is said to be \emph{$r$-step nilpotent} (for $r \geq 1$) if the $(r+1)$-fold simple commutator $[x_0,\ldots,x_r]$, defined inductively by $[x_0] = x_0$ and $[x_0,\ldots,x_i] = [[x_0,\ldots,x_{i-1}],x_i]$, is equal to the identity for all $x_0,\ldots,x_r \in G$. In this paper we adopt the convention that $[x,y] = x^{-1}y^{-1}xy$ for $x,y \in G$.

\begin{proof}[Proof of Corollary~\ref{c:ENfam}\ref{i:cENfam-nilp}]
Fix an integer $r \geq 1$, and let $\mathcal{G}$ be the family of all $r$-step nilpotent groups. Let $n \geq 1$, and let $S \subseteq F_n$ be a system of coefficient-free equations. As $F_n$ is countable, we may write $S = \{ s_1,s_2,\ldots \}$ for some $s_i \in F_n$, and define $S_i = \{ s_1,\ldots,s_i \}$ for any $i \geq 0$, so that $S = \bigcup_{i=0}^\infty S_i$.

Suppose for contradiction that for each $i \geq 0$, there exists $G_i \in \mathcal{G}$ such that $V_{G_i}(S_i) \neq V_{G_i}(S)$, and so there exists $\mathbf{g}^{(i)} = (g_1^{(i)},\ldots,g_n^{(i)}) \in G_i^n$ such that $s_j(g_1^{(i)},\ldots,g_n^{(i)}) = 1$ for $1 \leq j \leq i$ but not for all $j \geq 1$. By choosing the $\mathbf{g}^{(i)}$ one by one for $i = 1,2,\ldots$, and reordering the $\{ s_{i+1},s_{i+2},\ldots \}$ after each such choice if necessary, we may assume that $s_{i+1}(g_1^{(i)},\ldots,g_n^{(i)}) \neq 1$ for each $i \geq 0$. For each $i \geq 0$, let $\varphi_i: F_n \to G_i$ be the group homomorphism sending $X_j$ to $g^{(i)}_j$, where $\{ X_1,\ldots,X_n \}$ is a fixed free basis for $F_n$.

As $G_i$ is $r$-step nilpotent, it follows that any $(r+1)$-fold simple commutator maps to the identity under the map $\varphi_i$, and so $\ker(\varphi_i)$ contains $\gamma_{r+1}F_n$, the $(r+1)$-st term of the lower central series of $F_n$. Thus $\ker(\varphi_i) = p^{-1}(p(\ker\varphi_i))$, where $p: F_n \to F_n / \gamma_{r+1}F_n$ is the quotient map. By construction, we also know that $S_i \subseteq \ker(\varphi_i)$ but $s_{i+1} \notin \ker(\varphi_i)$ for each $i$, and so $p(S_i) \subseteq p(\ker(\varphi_i))$ but $p(s_{i+1}) \notin p(\ker(\varphi_i))$.

Thus, if $N_i \unlhd F_n/\gamma_{r+1}F_n$ is the normal closure of $p(S_i)$ in $F_n/\gamma_{r+1}F_n$, then $N_i \subseteq p(\ker(\varphi_i))$ but $N_{i+1} \nsubseteq p(\ker(\varphi_i))$ for each $i$, and so we have a strictly ascending chain $N_0 \lneq N_1 \lneq N_2 \lneq \cdots$ of subgroups of $F_n/\gamma_{r+1}F_n$. But as $F_n/\gamma_{r+1}F_n$ is $r$-step nilpotent and finitely generated, it is Noetherian (meaning that all its subgroups are finitely generated) -- this follows from the observation that $\gamma_iF / \gamma_{i+1}F$ is a finitely generated abelian group, generated by the images of $i$-fold simple commutators of the generators of $F$, for any finitely generated group $F$ and any $i \geq 1$. Thus no such strictly ascending chain of subgroups $N_0 \lneq N_1 \lneq N_2 \lneq \cdots$ can exist, and so for some $i$ we must have $V_G(S_i) = V_G(S)$ for all $G \in \mathcal{G}$. Hence $\mathcal{G}$ is an equationally Noetherian family, as required.
\end{proof}

The proof of Corollary~\ref{c:ENfam}\ref{i:cENfam-lin} given below relies on the Hilbert's basis theorem.

\begin{proof}[Proof of Corollary~\ref{c:ENfam}\ref{i:cENfam-lin}]
Fix integers $r,n \geq 1$, and consider the ring of integral polynomials in $r^2n$ commuting variables, $\mathcal{R} = \mathbb{Z}[X_{k,\ell}^{(m)} \mid 1 \leq k \leq r, 1 \leq \ell \leq r, 1 \leq m \leq n]$. Then it is easy to see -- by induction on the word length, say -- that for any $s \in F_n$ there exist polynomials $s_{i,j} \in \mathcal{R}$, $i,j \in \{1,\ldots,r\}$, such that
\[
s(A^{(1)},\ldots,A^{(n)}) = \begin{pmatrix} s_{1,1}(a^{(m)}_{k,\ell}) & \cdots & s_{1,r}(a^{(m)}_{k,\ell}) \\ \vdots & \ddots & \vdots \\ s_{r,1}(a^{(m)}_{k,\ell}) & \cdots & s_{r,r}(a^{(m)}_{k,\ell}) \end{pmatrix}
\]
for any matrices $A^{(m)} = (a^{(m)}_{k,\ell})_{k,\ell=1}^r \in GL_r(R)$, $1 \leq m \leq n$, and any commutative unital ring $R$. For $1 \leq i,j \leq r$, let $\widehat{s}_{i,j} = s_{i,j} - \delta_{i,j} \in \mathcal{R}$, where $\delta_{i,i} = 1 \in \mathbb{Z}$ and $\delta_{i,j} = 0$ for $i \neq j$. It then follows that for any $s \in F_n$, any commutative unital ring $R$ and any $A^{(m)} = (a^{(m)}_{k,\ell})_{k,\ell=1}^r \in GL_r(R)$, $1 \leq m \leq n$, we have $s(A^{(1)},\ldots,A^{(n)}) = 1$ in $GL_r(R)$ if and only if $\widehat{s}_{i,j} \in \ker(\Phi^{(R)}_{A^{(1)},\ldots,A^{(n)}})$ for $1 \leq i,j \leq r$, where $\Phi^{(R)}_{A^{(1)},\ldots,A^{(n)}}: \mathcal{R} \to R$ is a ring homomorphism sending $X^{(m)}_{k,\ell}$ to $a^{(m)}_{k,\ell}$ for each $k,\ell,m$.

Now let $\mathcal{G}$ be the family of all subgroups of $GL_r(R)$ for commutative unital rings $R$. Let $S = \{ s_1,s_2,\ldots \} \subseteq F_n$, and let $S_i = \{ s_1,\ldots,s_i \}$ for $i \geq 0$. Suppose for contradiction that for each $i \geq 0$ there exists $G_i \in \mathcal{G}$ (and so $G_i \leq GL_r(R_i)$) such that $V_{G_i}(S_i) \neq V_{G_i}(S)$. Then, in the same way as in the proof of Corollary~\ref{c:ENfam}\ref{i:cENfam-nilp} above, we may assume that there exists a commutative unital ring $R_i$ and a tuple $(A^{(i,1)},\ldots,A^{(i,n)}) \in (GL_r(R_i))^n$ for each $i \geq 0$ such that $s_j(A^{(i,1)},\ldots,A^{(i,n)}) = 1$ for $1 \leq j \leq i$ but $s_{i+1}(A^{(i,1)},\ldots,A^{(i,n)}) \neq 1$.

Now this implies, by the above, that for each $i \geq 0$ we have $\widehat{(s_j)}_{i',j'} \in \ker(\Phi^{(R_i)}_{A^{(i,1)},\ldots,A^{(i,n)}})$ for $1 \leq j \leq i$ and $1 \leq i',j' \leq r$, but $\widehat{(s_{i+1})}_{i',j'} \notin \ker(\Phi^{(R_i)}_{A^{(i,1)},\ldots,A^{(i,n)}})$ for some $i',j' \in \{1,\ldots,r \}$. Thus, if $I_i$ is the ideal of $\mathcal{R}$ generated by $\widehat{(s_j)}_{i',j'}$ for $1 \leq j \leq i$ and $1 \leq i',j' \leq r$, then for each $i \geq 0$ the ideal $I_{i+1}$ strictly contains $I_i$. This gives rise to a strictly ascending chain of ideals $I_0 \subsetneq I_1 \subsetneq I_2 \subsetneq \cdots$ of $\mathcal{R}$. But as $\mathbb{Z}$ is Noetherian and $\mathcal{R}$ is a ring of polynomials in commuting variables over $\mathbb{Z}$, this contradicts the Hilbert's basis theorem. Thus, for some $i$, we must have $V_G(S_i) = V_G(S)$ for all $G \in \mathcal{G}$, and so $\mathcal{G}$ is an equationally Noetherian family, as required.
\end{proof}

\begin{rmk}
The same proof, just by using the ring of polynomials over $R$ instead of over $\mathbb{Z}$, shows that if $R$ is a \emph{Noetherian} commutative unital ring, then $GL_r(R)$ is \emph{strongly} equationally Noetherian, recovering \cite[Theorem B1]{bmr}. However, as shown in the next result, this need not hold if $R$ is not assumed to be Noetherian.
\end{rmk}

\begin{lem} \label{l:nnsen}
There exists a $2$-step nilpotent group $G$ that is linear (over a commutative unital ring) but not strongly equationally Noetherian.
\end{lem}

\begin{proof}
Let $R = \mathbb{Z}[X_0,X_1,\ldots]/I$, where $I \lhd \mathbb{Z}[X_0,X_1,\ldots]$ is the ideal generated by $\{X_0-1\} \cup \{X_{j+1}^2-X_j \mid j \geq 0\}$. 
Let $G = UT_3(R) \leq GL_3(R)$ be the group of upper unitriangular $3 \times 3$ matrices with coefficients in $R$; it is clear that $G$ is linear and $2$-step nilpotent.

Now for each $j \geq 0$, let
\[
A_j = \begin{pmatrix} 1&1&0\\&1&X_j\\&&1 \end{pmatrix} \in G \qquad \text{and} \qquad B_j = \begin{pmatrix} 1&b_j&0\\&1&b_j\\&&1 \end{pmatrix} \in G,
\]
where $b_j = \sum_{\ell=0}^{2^j-1} X_j^\ell \in R$. It is easy to check that $[A_j,B_k] = 1$ if and only if $b_k = X_jb_k$ in $R$. Now since we have $X_j^{2^j} = 1$ in $R$ for each $j \geq 0$, it follows that we have $b_k = X_jb_k$ if $j \leq k$. On the other hand, if $j > k$ then we have $\Psi_j(b_k) = 2^k \neq -2^k = \Psi_j(X_jb_k)$, where $\Psi_j: R \to \mathbb{C}$ is a ring homomorphism defined by $\Psi_j(X_\ell) = \exp(2^{j-\ell}\pi i)$; therefore, $b_k \neq X_j b_k$. Thus $[A_j,B_k] = 1$ if and only if $j \leq k$.

Now for each $j \geq 0$, let $s_j = [A_j,Y] \in G * F_1(Y)$ and $S_j = \{ s_0,\ldots,s_j \} \subset G * F_1(Y)$, and let $S = \{s_0,s_1,\ldots\} = \bigcup_{j=0}^\infty S_j$. Then, as shown above, we have $s_j(B_k) = 1$ in $G$ if and only if $j \leq k$, and so $B_j \in V_G(S_j) \setminus V_G(S)$ for each $j$. Thus $V_G(S') \neq V_G(S)$ for any finite subset $S' \subset S$, and so $G$ is not strongly equationally Noetherian, as required.
\end{proof}

\section{Abelian-by-polycyclic groups} \label{s:abpolyc}

In this section we prove Theorem~\ref{t:abpolyc}. We first reduce to systems of `positive' or `exponent sum zero' equations, as follows. Given an equation $s \in F_n(X_1,\ldots,X_n) * G$ for some group $G$ and some $n \geq 1$, we say that:
\begin{enumerate}[label=(\roman*)]
\item $s$ is \emph{positive} if it is contained in the submonoid of $F_n * G$ generated (as a monoid) by $\{ X_1,\ldots,X_n \} \cup G$;
\item $s$ has \emph{exponent sum zero} if $\Psi(s) = 0$, where $\Psi: F_n * G \to F_n \to F_n/[F_n,F_n] \cong \mathbb{Z}^n$ is the canonical quotient map.
\end{enumerate}
In our proof of Theorem~\ref{t:abpolyc}, we use the equivalence of \ref{i:l+} and \ref{i:l+0-en} in the following result. We record the equivalence of \ref{i:l0} and \ref{i:l+0-en} for future reference, as we will use it in Section \ref{ss:ENnotRF}.

\begin{lem} \label{l:+0}
For any group $G$, the following are equivalent:
\begin{enumerate}[label={\textup{(\roman*)}}]
\item \label{i:l+} For any system $S \subseteq F_n * G$ of positive equations (respectively positive coefficient-free equations), there exists a finite subsystem $S_0 \subseteq S$ such that $V_G(S_0) = V_G(S)$.
\item \label{i:l0} For any system $S \subseteq F_n * G$ of equations of exponent sum zero (respectively coefficient-free equations of exponent sum zero), there exists a finite subsystem $S_0 \subseteq S$ such that $V_G(S_0) = V_G(S)$.
\item \label{i:l+0-en} $G$ is strongly equationally Noetherian (respectively equationally Noetherian).
\end{enumerate}
\end{lem}

\begin{proof}
The implications \ref{i:l+0-en}~$\Rightarrow$~\ref{i:l+} and \ref{i:l+0-en}~$\Rightarrow$~\ref{i:l0} are trivial. We will show the implications \ref{i:l+}~$\Rightarrow$~\ref{i:l+0-en} and \ref{i:l0}~$\Rightarrow$~\ref{i:l+0-en} for systems of equations that need not be coefficient-free; the proof for coefficient-free equations is almost identical.

For \ref{i:l+}~$\Rightarrow$~\ref{i:l+0-en}, let $n \geq 1$, let $S \subseteq F_n(X_1,\ldots,X_n) * G$ be a system of equations, and suppose that \ref{i:l+} holds. Let $\Phi: F_n(X_1,\ldots,X_n) * G \to F_{2n}(X_1,\ldots,X_{2n}) * G$ be an injective monoid homomorphism defined by $\Phi(X_i) = X_i$, $\Phi(X_i^{-1}) = X_{n+i}$ for $1 \leq i \leq n$, and $\Phi(g) = g$ for any $g \in G$. Note that $\Phi$ sends any equation to a positive equation, and that
\begin{equation} \label{e:l+}
V_G(T) = \{ (g_1,\ldots,g_n) \in G^n \mid (g_1,\ldots,g_n,g_1^{-1},\ldots,g_n^{-1}) \in V_G(\Phi(T)) \}
\end{equation}
for any $T \subseteq F_n * G$. But as $\Phi(S)$ is a system of positive equations, it follows from \ref{i:l+} that $V_G(\Phi(S)) = V_G(\Phi(S_0))$ for a finite subset $S_0 \subseteq S$, and so $V_G(S_0) = V_G(S)$ by \eqref{e:l+}. Thus $G$ is strongly equationally Noetherian, as required.

For \ref{i:l0}~$\Rightarrow$~\ref{i:l+0-en}, let $n \geq 1$, let $S \subseteq F_n * G$ be a system of equations, and suppose that \ref{i:l0} holds. Let $\Psi: F_n * G \to \mathbb{Z}^n$ be the quotient map as above. Then the subgroup $\langle \Psi(S) \rangle \leq \mathbb{Z}^n$ is finitely generated, and so there exists a finite subset $S_0 = \{ s_1,\ldots,s_m \} \subseteq S$ such that $\langle \Psi(S_0) \rangle = \langle \Psi(S) \rangle$. We may then define a system of equations $S' = \{ s_1^{\alpha_1(s)} \cdots s_m^{\alpha_m(s)} s \mid s \in S \setminus S_0 \}$, where the $\alpha_i(s) \in \mathbb{Z}$ are chosen in such a way that $\Psi(S') = \{0\}$ -- that is, $S'$ is a system of exponent sum zero equations. It follows from \ref{i:l0} that $V_G(S_1') = V_G(S')$ for a finite subset $S_1' \subseteq S'$, and we may define $S_1 \subseteq S \setminus S_0$ to be a finite subset such that $S_1' = \{ s_1^{\alpha_1(s)} \cdots s_m^{\alpha_m(s)} s \mid s \in S_1 \}$. We then have
\[
V_G(S) = V_G(S_0) \cap V_G(S') = V_G(S_0) \cap V_G(S_1') = V_G(S_0) \cap V_G(S_1) = V_G(S_0 \cup S_1),
\]
and so $G$ is strongly equationally Noetherian, as required.
\end{proof}



\begin{proof}[Proof of Theorem~\ref{t:abpolyc}]
Let $G$ be an abelian-by-polycyclic group: namely, there exists a normal abelian subgroup $A \unlhd G$ such that $P = G/A$ is polycyclic. Then $G$ is a subgroup of the \emph{unrestricted} wreath product $A^P \rtimes P$ (see \cite[Theorem 22.21]{hneumann}), which is a split extension of an abelian group $A^P$ by a polycyclic group $P$. As subgroups of equationally Noetherian groups are equationally Noetherian, we may therefore assume, without loss of generality, that $G$ is a \emph{split} extension of $A$ by $P$, i.e.\ $G = A \rtimes P$. Note that $P$ acts on $A$ on the right by conjugation, making $A$ into a right $\mathbb{Z}P$-module.

Let $s \in F_n(X_1,\ldots,X_n)$ be a positive coefficient-free equation. Then $s = Y_1 \cdots Y_k$ for some $Y_i \in \{ X_1,\ldots,X_n \}$. We can write $s(a_1p_1,\ldots,a_np_n) = \hat{a}_1\hat{p}_1 \cdots \hat{a}_k\hat{p}_k$, where $a_j,\hat{a}_i \in A$, $p_j,\hat{p}_i \in P$, and where $\hat{a}_i\hat{p}_i = a_jp_j$ if $Y_i = X_j$. We may rewrite this as
\[
s(a_1p_1,\ldots,a_np_n) = s(p_1,\ldots,p_n) \hat{a}_1^{\hat{p}_1 \cdots \hat{p}_k} \hat{a}_2^{\hat{p}_2 \cdots \hat{p}_k} \cdots \hat{a}_k^{\hat{p}_k}.
\]
Writing $A$ additively as a $\mathbb{Z}P$-module, we may deduce that
\begin{equation} \label{e:solnAbyP}
s(a_1p_1,\ldots,a_np_n) = 1 \qquad \Leftrightarrow \qquad s(p_1,\ldots,p_n) = 1 \text{ and } \sum_{j=1}^n a_j \cdot s_j(p_1,\ldots,p_n) = 0,
\end{equation}
where
\[
s_j(p_1,\ldots,p_n) = \sum_{\substack{1 \leq i \leq k \\ Y_i = X_j}} p_i \cdots p_k \in \mathbb{Z}P.
\]
Note that we have $s_j(p_1,\ldots,p_n) = \Psi_{p_1,\ldots,p_n}(\bar{s}_j)$, where
\[
\bar{s}_j = \sum_{\substack{1 \leq i \leq k \\ Y_i = X_j}} Y_i \cdots Y_k \in \mathbb{Z}F_n,
\]
and $\Psi_{p_1,\ldots,p_n}: \mathbb{Z}F_n \to \mathbb{Z}P$ is a ring homomorphism induced by the group homomorphism $\psi_{p_1,\ldots,p_n}: F_n \to P$ sending $X_j$ to $p_j$.

Now as $P$ is polycyclic and $F_n$ is finitely generated, it is known (see \cite[Lemma 3.1]{gw1} and \cite[Example 2.2]{gw2}) that there exists a polycyclic group $Q$ and a surjective group homomorphism $\varepsilon: F_n \to Q$ such that any group homomorphism $\psi: F_n \to P$ factors through $\varepsilon$; let $E: \mathbb{Z}F_n \to \mathbb{Z}Q$ be the ring homomorphism induced by $\varepsilon$. In particular, for each $(p_1,\ldots,p_n) \in P^n$ there is a group homomorphism $\gamma_{p_1,\ldots,p_n}: Q \to P$ such that $\psi_{p_1,\ldots,p_n} = \gamma_{p_1,\ldots,p_n} \circ \varepsilon$, inducing a ring homomorphism $\Gamma_{p_1,\ldots,p_n}: \mathbb{Z}Q \to \mathbb{Z}P$ such that $\Psi_{p_1,\ldots,p_n} = \Gamma_{p_1,\ldots,p_n} \circ E$. Thus we can rewrite \eqref{e:solnAbyP} as
\begin{equation} \label{e:solnAbyP2}
s(a_1p_1,\ldots,a_np_n) = 1 \qquad \Leftrightarrow \qquad s(p_1,\ldots,p_n) = 1 \text{ and } \sum_{j=1}^n a_j \cdot \Gamma_{p_1,\ldots,p_n}(\hat{s}_j) = 0,
\end{equation}
where $\hat{s}_j = E(\bar{f}_j) \in \mathbb{Z}Q$. Moreover, as $Q$ is polycyclic, it is known that the ring $\mathbb{Z}Q$ is Noetherian, i.e.\ any submodule of a finitely generated right $\mathbb{Z}Q$-module is finitely generated \cite[Theorem 1]{hallpolycyc}.

Now let $S \subseteq F_n$ be a system of positive coefficient-free equations: it is enough to consider such a system by Lemma~\ref{l:+0}. As $P$ is polycyclic, and therefore linear, it follows that $P$ is strongly equationally Noetherian. In particular, $V_P(S) = V_P(S_1)$ for some finite subset $S_1 \subseteq S$. Moreover, as the ring $\mathbb{Z}Q$ is Noetherian, it follows that $M(S) = M(S_2)$ for some finite subset $S_2 \subseteq S$, where $M(T)$ is the submodule of the right $\mathbb{Z}Q$-module $(\mathbb{Z}Q)^n$ generated by $\{ (\hat{t}_1,\ldots,\hat{t}_n) \mid t \in T \}$ (for any $T \subseteq F_n$). It follows that we also have $M_{p_1,\ldots,p_n}(S) = M_{p_1,\ldots,p_n}(S_2)$ for any $(p_1,\ldots,p_n) \in P^n$, where $M_{p_1,\ldots,p_n}(T) = \{ (\Gamma_{p_1,\ldots,p_n}(g_1),\ldots,\Gamma_{p_1,\ldots,p_n}(g_n)) \mid (g_1,\ldots,g_n) \in M(T) \}$ for any $T \subseteq F_n$. Therefore, it follows from \eqref{e:solnAbyP2} that, for any $a_j \in A$ and $p_j \in P$,
\begin{align*}
(a_1p_1,\ldots,a_np_n) \in V_G(S) \qquad &\Leftrightarrow \qquad (p_1,\ldots,p_n) \in V_P(S) \\ &\qquad\quad \text{and } \sum_{j=1}^n a_j \cdot h_j = 0 \text{ for all } (h_1,\ldots,h_n) \in M_{p_1,\ldots,p_n}(S) \\ &\Leftrightarrow \qquad (p_1,\ldots,p_n) \in V_P(S_1) \\ &\qquad\quad \text{and } \sum_{j=1}^n a_j \cdot h_j = 0 \text{ for all } (h_1,\ldots,h_n) \in M_{p_1,\ldots,p_n}(S_2).
\end{align*}
In particular, $V_G(S) = V_G(S_1 \cup S_2)$. Thus, by Lemma~\ref{l:+0}, $G$ is equationally Noetherian, as required.
\end{proof}

\section{\texorpdfstring{$\mathcal{G}$}{G}-limit groups} \label{s:limit}

Recall that an \emph{ultrafilter} (on $\mathbb{N}$) is a finitely additive probability measure $\omega: 2^{\mathbb{N}} \to \{0,1\}$. An ultrafilter $\omega$ is said to be \emph{non-principal} if $\omega(F) = 0$ for every finite subset $F \subset \mathbb{N}$; non-principal ultrafilters exist subject to the Axiom of Choice. Throughout this section, we fix a non-principal ultrafilter $\omega$. Given a sequence $(P_j)_{j \in \mathbb{N}}$ of statements, we say \emph{$P_j$ holds $\omega$-almost surely} if $\omega(\{ j \in \mathbb{N} \mid P_j \text{ holds} \}) = 1$.

Let $\mathcal{G}$ be a family of groups. Given a group $G$, we write $\Hom(G,\mathcal{G})$ for the set of homomorphisms $\varphi: G \to H$ with $H \in \mathcal{G}$. The \emph{$\omega$-kernel} of a sequence of homomorphisms $(\varphi_j)_{j \in \mathbb{N}}$ from $\Hom(G,\mathcal{G})$ is defined as
\[
\ker^\omega(\varphi_j) = \{ g \in G \mid \varphi_j(g) = 1 \text{ $\omega$-almost surely} \}.
\]
It is easy to see that $\ker^\omega(\varphi_j)$ is a normal subgroup of $G$.

\begin{defn}
Let $\mathcal{G}$ be a family of groups. A \emph{$\mathcal{G}$-limit group} is a group of the form $G/\ker^\omega(\varphi_j)$ for a finitely generated group $G$ and a sequence $(\varphi_j)_{j \in \mathbb{N}}$ in $\Hom(G,\mathcal{G})$.
\end{defn}

We start by proving (a strengthening of) the first part of Proposition~\ref{p:limit}.

\begin{lem} \label{l:limit-EN}
Let $\mathcal{G}$ be an equationally Noetherian family of groups, and let $\mathcal{L}$ be a family of $\mathcal{G}$-limit groups. Then $\mathcal{L}$ is equationally Noetherian.
\end{lem}

\begin{proof}
Let $S \subseteq F_n$ be a system of equations. Since $\mathcal{G}$ is equationally Noetherian, there exists a finite subset $S_0 \subseteq S$ such that $V_G(S_0) = V_G(S)$ for all $G \in \mathcal{G}$. We claim that $V_L(S_0) = V_L(S)$ for all $L \in \mathcal{L}$.

Thus, let $L \in \mathcal{L}$, so that $L = G / \ker^\omega(\varphi_j)$ for a group $G$ and homomorphisms $\varphi_j: G \to G_j$ (where $G_j \in \mathcal{G}$) for all $j \in \mathbb{N}$. We write $\varphi_\infty: G \to L$ for the quotient map, and $\varphi_\infty^{(n)}$ for the induced quotient map from $G^n$ to $L^n$. Since $\varphi(s(g_1,\ldots,g_n)) = s(\varphi(g_1),\ldots,\varphi(g_n))$ for every $s \in S$, $\varphi \in \Hom(G,\mathcal{G} \cup \{L\})$ and $g_1,\ldots,g_n \in G$, we have
\begin{align*}
V_L(S) &= \{ (h_1,\ldots,h_n) \in L^n \mid s(h_1,\ldots,h_n) = 1 \text{ for all } s \in S \} \\
&= \varphi_\infty^{(n)} \left( \{ (g_1,\ldots,g_n) \in G^n \mid s(g_1,\ldots,g_n) \in \ker(\varphi_\infty) \text{ for all } s \in S \} \right) \\
&= \varphi_\infty^{(n)} \left( \{ (g_1,\ldots,g_n) \in G^n \mid \text{for all } s \in S, \varphi_j(s(g_1,\ldots,g_n)) = 1 \text{ $\omega$-almost surely} \} \right) \\
&\supseteq \varphi_\infty^{(n)} \left( \{ (g_1,\ldots,g_n) \in G^n \mid \text{$\omega$-almost surely}, \varphi_j(s(g_1,\ldots,g_n)) = 1  \text{ for all } s \in S \} \right) \\
&= \varphi_\infty^{(n)} \left( \{ (g_1,\ldots,g_n) \in G^n \mid \text{$\omega$-almost surely}, (\varphi_j(g_1),\ldots,\varphi_j(g_n)) \in V_{G_j}(S) \} \right).
\end{align*}
Moreover, since $\omega$ is finitely additive, the inclusion ($\supseteq$) becomes an equality when $S$ is finite. As $V_{G_j}(S) = V_{G_j}(S_0)$ for all $j \in \mathbb{N}$, we thus have
\begin{align*}
V_L(S) &\supseteq \varphi_\infty^{(n)} \left( \{ (g_1,\ldots,g_n) \in G^n \mid \text{$\omega$-almost surely}, (\varphi_j(g_1),\ldots,\varphi_j(g_n)) \in V_{G_j}(S) \} \right) \\
&= \varphi_\infty^{(n)} \left( \{ (g_1,\ldots,g_n) \in G^n \mid \text{$\omega$-almost surely}, (\varphi_j(g_1),\ldots,\varphi_j(g_n)) \in V_{G_j}(S_0) \} \right) \\
&= V_L(S_0) \supseteq V_L(S)
\end{align*}
and so $V_L(S) = V_L(S_0)$, as claimed.
\end{proof}

\begin{rmk}
The proof of Lemma~\ref{l:limit-EN} does not use the fact that the group $G$ is finitely generated, and so we have actually proved something stronger. For instance, the same proof shows that if $\mathcal{G}$ is an equationally Noetherian family, then the family of all (not necessarily finitely generated) direct limits $\varinjlim G_j$ (with $G_j \in \mathcal{G}$) is equationally Noetherian. In general, a direct limit of finite groups that belong to an equationally Noetherian family need not be residually finite, as the (divisible) quasicyclic group $\mathbb{Z}\left[\frac{1}{2}\right] \big/ \mathbb{Z} \cong \varinjlim C_{2^j}$ shows; the family $\{ C_{2^j} \mid j \in \mathbb{N} \}$ is equationally Noetherian by Corollary~\ref{c:ENfam}\ref{i:cENfam-nilp}.
\end{rmk}

\begin{lem} \label{l:limit-RF}
Let $\mathcal{G}$ be an equationally Noetherian family of finite groups, and let $L$ be a $\mathcal{G}$-limit group. Then $L$ is residually finite.
\end{lem}

\begin{proof}
We have $L = G / \ker^\omega(\varphi_j)$ for a finitely generated group $G$ and homomorphisms $\varphi_j: G \to G_j$ (where $G_j \in \mathcal{G}$) for all $j \in \mathbb{N}$. Let $\varphi_\infty: G \to L$ be the quotient map.

Let $h \in L$ be a non-trivial element, and let $g \in \varphi_\infty^{-1}(h)$. Thus $\varphi_j(g) \neq 1$ $\omega$-almost surely. Moreover, as $\mathcal{G}$ is equationally Noetherian and $G$ is finitely generated, it follows from \cite[Theorem 3.6]{gh} that $\ker(\varphi_\infty) \subseteq \ker(\varphi_j)$ $\omega$-almost surely. Thus there exists $j \in \mathbb{N}$ such that we have both $\varphi_j(g) \neq 1$ and $\ker(\varphi_\infty) \subseteq \ker(\varphi_j)$. The latter implies that $\varphi_j = \eta \circ \varphi_\infty$ for a homomorphism $\eta: L \to G_j$; moreover, we have $\eta(h) = \eta(\varphi_\infty(g)) = \varphi_j(g) \neq 1$. As $G_j$ is finite, this shows that $L$ is residually finite, as required.
\end{proof}

\begin{proof}[Proof of Proposition~\ref{p:limit}]
This follows immediately from Lemmas \ref{l:limit-EN} and \ref{l:limit-RF}.
\end{proof}

\section{Cyclic subgroup separability of wreath products} \label{s:cssep}

In this section we prove Theorem~\ref{t:css}.

Throughout the section, we employ the following notation. Given groups $G$ and $H$, we write elements of the (regular, restricted) wreath product $G \wr H$ as $(f,a)$, where $a \in H$ and $f: H \to G$ is a function such that $f(x) = 1$ for all but finitely many $x \in H$. The \emph{support} of such a function $f: H \to G$ is the finite subset $\supp(f) = \{ x \in H \mid f(x) \neq 1 \}$ of $H$. Given $f: H \to G$ and $a \in H$, we also write $f_a: H \to G$ for the function defined by $f_a(x) = f(xa^{-1})$. With this notation, the group operations in $G \wr H$ become $(f,a)(g,b) = (f_bg,ab)$ and $(f,a)^{-1} = (f_{a^{-1}}^{-1},a^{-1})$.

\begin{lem} \label{l:Acss}
Let $G$ be a group that is cyclic subgroup separable, and let $\mathcal{X}$ be a set. If $f,h \in \bigoplus_{x \in \mathcal{X}} G$ are elements such that $h \notin \langle f \rangle$, then there exists a finite-index normal subgroup $B \lhd G$ such that $\pi_B(h) \notin \langle \pi_B(f) \rangle$, where $\pi_B: \bigoplus_{x \in \mathcal{X}} G \to \bigoplus_{x \in \mathcal{X}} (G/B)$ is the canonical quotient map.
\end{lem}

\begin{proof}
We will write elements of $\bigoplus_{x \in \mathcal{X}} G$ as functions $g: \mathcal{X} \to G$ with finite support, with the group operations corresponding to pointwise multiplication and inversion. We will split the argument into two cases, depending on whether or not $f(x)$ has finite order for all $x \in \mathcal{X}$.

\begin{description}

\item[If $f(x)$ has finite order for all $x \in \mathcal{X}$] In this case, as $f$ has finite support, it follows that there exists an integer $p \geq 1$ such that $f(x)^p = 1$ for all $x \in \mathcal{X}$, i.e.\ $f^p = 1$. As $h \notin \langle f \rangle$, it follows that for each $i = 0,\ldots,p-1$, there exists $x_i \in \mathcal{X}$ such that $h(x_i) \neq f(x_i)^i$. Since $G$ is cyclic subgroup separable, and so residually finite, there exists a finite-index normal subgroup $B \lhd G$ such that $h(x_i) f(x_i)^{-i} \notin B$ for all $i$. But this implies that $\pi_B(h) \notin \{ \pi_B(f^i) \mid 0 \leq i \leq p-1 \} = \langle \pi_B(f) \rangle$, as required.

\item[If $f(x_0)$ has infinite order for some $x_0 \in \mathcal{X}$] If $h(x_0) \notin \langle f(x_0) \rangle$, then as $G$ is cyclic subgroup separable, there exists a finite-index normal subgroup $B \lhd G$ such that $h(x_0)B \notin \langle f(x_0)B \rangle$. Therefore, $\pi_B(h) \notin \langle \pi_B(f) \rangle$, and we are done. Otherwise, there exists a (unique, since $f(x_0)$ has infinite order) integer $p_0 \in \mathbb{Z}$ such that $h(x_0) = f(x_0)^{p_0}$. As $h \notin \langle f \rangle$, there exists $x_1 \in \mathcal{X}$ such that $h(x_1) \neq f(x_1)^{p_0}$.

Now define an integer $r \geq 1$ as follows. If $f(x_1)$ has finite order, let $r$ be this order. If $f(x_1)$ has infinite order and $h(x_1) \notin \langle f(x_1) \rangle$, let $r \geq 1$ be any integer. Otherwise, there exists a unique $p_1 \in \mathbb{Z}$, $p_1 \neq p_0$, such that $h(x_1) = f(x_1)^{p_1}$; in this case, let $r \geq 1$ be any integer not dividing $p_0-p_1$. It then follows from the definition of $r$ that for any $p \in \mathbb{Z}$, we have $h(x_1)f(x_1)^{-p} \notin \langle f(x_1)^r \rangle$ if $p \equiv p_0 \pmod{r}$, and $h(x_0)f(x_0)^{-p} \notin \langle f(x_0)^r \rangle$ otherwise.

Let $p' \in \{ 0,\ldots,r-1 \}$ be the integer such that $p_0 \equiv p' \pmod{r}$. As $G$ is cyclic subgroup separable, there exists a finite-index normal subgroup $B \lhd G$ such that we have $h(x_1)f(x_1)^{-p'}B \notin \langle (f(x_1)B)^r \rangle$, as well as $h(x_0)f(x_0)^{-p}B \notin \langle (f(x_0)B)^r \rangle$ for all $p \in \{ 0,\ldots,r-1 \} \setminus \{ p' \}$. Therefore, for all $p \in \mathbb{Z}$ we have either $h(x_1)B \neq (f(x_1)B)^p$ or $h(x_0)B \neq (f(x_0)B)^p$, and so $\pi_B(h) \neq \pi_B(f)^p$. Thus $\pi_B(h) \notin \langle \pi_B(f) \rangle$, as required. \qedhere

\end{description}
\end{proof}

\begin{prop} \label{p:AGcss}
Let $G$ be an abelian group, let $H$ be a group, and suppose that $G$ and $H$ are cyclic subgroup separable. Then $G \wr H$ is cyclic subgroup separable.
\end{prop}

\begin{proof}
Throughout this proof, we write the group operation in $G$ additively. It is easy to check that
\[
(f,a)^n = (g,a^n)
\]
for any $n \in \mathbb{Z}$, where
\[
g = \begin{cases} - (f_{a^{-1}} + \cdots + f_{a^n}) & \text{if } n < 0, \\ 0 & \text{if } n = 0, \\ f + f_a + \cdots + f_{a^{n-1}} & \text{if } n > 0. \end{cases}
\]
In particular, if $a^n = 1$ in $H$, then $(f,a)^n = (g,1)$ and $g: H \to G$ is invariant under right multiplication by $a$, i.e.\ $g(x) = g(xa)$ for all $x \in H$.

Let $(f,a)$ and $(h,c)$ be two elements of $G \wr H$, and suppose that $(h,c)$ is not a power of $(f,a)$. If $c \notin \langle a \rangle$, then, as $H$ is cyclic subgroup separable, there exist a finite-index normal subgroup $N \lhd H$ such that $cN \notin \langle aN \rangle$, and so we are done. Thus, we may assume, without loss of generality, that $c = a^m$ for some $m \in \mathbb{Z}$. We will consider two cases, depending on whether or not $a$ has finite order in $H$.

\begin{description}

\item[If $a$ has order $n < \infty$ in $H$] In this case, we have $(f,a)^n = (g,1)$ for some function $g: H \to G$ of finite support. Without loss of generality, we may assume that $0 \leq m \leq n-1$; let $k = h-f-f_a-\cdots-f_{a^{m-1}}: H \to G$. As $H$ is cyclic subgroup separable, it is residually finite, and so there exists a finite-index normal subgroup $N \lhd H$ such that the restriction of the quotient map $H \to H/N$ to $\supp(k) \cup \supp(g) \cup \langle a \rangle$ is injective. 

Moreover, as for any $q \in \mathbb{Z}$ we have
\[
(h,a^m) = (h,c) \neq (f,a)^{qn+m} = (f+f_a+\cdots+f_{a^{m-1}}+qg,a^m) = (h-k+qg,a^m),
\]
it follows that $k \notin \langle g \rangle$. By Lemma~\ref{l:Acss}, there exists a finite-index normal subgroup $B \lhd G$ such that the image of $(k,1)$ in $(G/B) \wr H$ is not a power of the image of $(g,1)$. By the choice of $N$, it thus follows that $\pi(k,1) \notin \langle \pi(g,1) \rangle$, where $\pi: G \wr H \to (G/B) \wr (H/N)$ is the quotient map.

We claim that $\pi(h,c) \notin \langle \pi(f,a) \rangle$ -- that is, $\pi(h,c) \neq \pi(f,a)^p$ for any $p$. Indeed, let $p \in \mathbb{Z}$. If $p \not\equiv m \pmod{n}$, then $cN \neq (aN)^p$ and we are done. Otherwise, we have $p = nq+m$ for some $q \in \mathbb{Z}$, and so $(f,a)^p = (h-k+qg,a^m)$. By the choice of $B$ and $N$, we have $\pi(-k+qg,1) \neq 1$, hence $\pi(h,c) \neq \pi(f,a)^p$ as well. Thus, $\pi(h,c) \notin \langle \pi(f,a) \rangle$, as claimed -- therefore, as $(G/B) \wr (H/N)$ is finite, we are done.

\item[If $a$ has infinite order in $H$] Let $\hat{f}: H \to G$ be such that $(f,a)^m = (\hat{f},a^m)$. As $H$ is cyclic subgroup separable, and so residually finite, and as $\supp(\hat{f})$ and $\supp(h)$ are finite, there exists a finite-index normal subgroup $N \lhd H$ such that the restriction of the quotient map $H \to H/N$ to $\supp(\hat{f}) \cup \supp(h) \cup \{ 1,a,\ldots,a^{|\supp(\hat{f})| + |\supp(h)|} \}$ is injective. In particular, since $a$ has infinite order, $aN$ has order $n > |\supp(\hat{f})|+|\supp(h)|$ in $H/N$. Moreover, as $(h,a^m) = (h,c) \neq (f,a)^m = (\hat{f},a^m)$, there exists $x_0 \in H$ such that $h(x_0) \neq \hat{f}(x_0)$. As $G$ is cyclic subgroup separable, and so residually finite, there exists $B \lhd G$ such that $h(x_0)+B \neq \hat{f}(x_0)+B$.

We claim that $\pi(h,c) \notin \langle \pi(f,a) \rangle$, where $\pi: G \wr H \to (G/B) \wr (H/N)$ is the quotient map as before -- that is, we claim that $\pi(h,c) \neq \pi(f,a)^p$ for any $p$. Indeed, let $p \in \mathbb{Z}$. If $p \not\equiv m \pmod{n}$, then $cN \neq (aN)^p$ and we are done. Otherwise, we have $p = nq+m$ for some $q \in \mathbb{Z}$, and so $\pi(f,a)^p = (\hat{f}^{(N)}+qg^{(N)},(aN)^m)$, where $g: H \to G$ is such that $(f,a)^n = (g,a^n)$. Now if $qg^{(N)}(xN) = 0$ for all $xN \in H/N$, then $\pi(f,a)^p = (\hat{f}^{(N)},(aN)^m) \neq (h^{(N)},cN) = \pi(h,c)$ since $\hat{f}^{(N)}(x_0N) = \hat{f}(x_0)+B \neq h(x_0)+B = h^{(N)}(x_0N)$ by the choices of $N$ and $B$, and so we are done. Thus, we may assume that $qg^{(N)}(xN) \neq 0$ for some $xN \in H/N$. But as $aN$ has order $n$ in $H/N$ and so $(f^{(N)},aN)^n = (g^{(N)},1)$, it follows that $g^{(N)}(xN) = g^{(N)}((xN)(aN))$ for all $xN \in H/N$, and so $|\supp(g^{(N)})|$ (as well as $|\supp(qg^{(N)})|$) is divisible by $n$: in particular,
\[
|\supp(qg^{(N)})| \geq n > |\supp(\hat{f})|+|\supp(h)| \geq |\supp(\hat{f}^{(N)})|+|\supp(h^{(N)})|.
\]
But this implies that $\hat{f}^{(N)}+qg^{(N)}$ and $h^{(N)}$ have different supports, and so $\pi(f,a)^p \neq \pi(h,c)$. Thus $\pi(h,c) \notin \langle \pi(f,a) \rangle$, as claimed, and so, as $(G/B) \wr (H/N)$ is finite, we are done. \qedhere

\end{description}
\end{proof}

\begin{lem} \label{l:fincss}
Let $G$ and $H$ be groups with $H$ finite. Then $G \wr H$ is cyclic subgroup separable if and only if $G$ is cyclic subgroup separable.
\end{lem}

\begin{proof}
Since $G$ is a subgroup of $G \wr H$, the $(\Rightarrow)$ direction is immediate.

For $(\Leftarrow)$, let $(f,a),(h,c) \in G \wr H$ be elements such that $(h,c) \notin \langle (f,a) \rangle$. If $c \notin \langle a \rangle$, then the image of $(h,c)$ in the finite quotient $H$ of $G \wr H$ is not a power of the image of $(f,a)$, and so we are done. Thus, we may assume that $c = a^m$ for some $m \in \mathbb{Z}$. Let $n < \infty$ be the order of $a$ in $H$; without loss of generality, assume that $0 \leq m \leq n-1$.

Let $g: H \to G$ be such that $(f,a)^n = (g,1)$. Given any $q \in \mathbb{Z}$, we have
\[
(h,a^m) = (h,c) \neq (f,a)^{qn+m} = (f_{a^{m-1}} \cdots f_afg^q,a^m)
\]
and so $k \notin \langle g \rangle$, where $k = f^{-1} f_a^{-1} \cdots f_{a^{m-1}}^{-1}h$. By Lemma~\ref{l:Acss}, it follows that there exists a finite-index normal subgroup $B \lhd G$ such that $\pi(k,1) \notin \langle \pi(g,1) \rangle$, where $\pi: G \wr H \to (G/B) \wr H$ is the quotient map.

We claim that $\pi(h,c) \notin \langle \pi(f,a) \rangle$ -- that is, $\pi(h,c) \neq \pi(f,a)^p$ for any $p \in \mathbb{Z}$. Indeed, let $p \in \mathbb{Z}$. If $p \not\equiv m \pmod{n}$, then $c \neq a^p$ and we are done. Otherwise, we have $p = nq+m$ for some $q \in \mathbb{Z}$, and so $(f,a)^p = (hk^{-1}g^q,a^m)$. By the choice of $B$, we have $\pi(k^{-1}g^q,1) \neq 1$, hence $\pi(h,c) \neq \pi(f,a)^p$ as well. Thus, $\pi(h,c) \notin \langle \pi(f,a) \rangle$, as claimed -- therefore, as $(G/B) \wr H$ is finite, we are done.
\end{proof}

\begin{proof}[Proof of Theorem~\ref{t:css}]
Suppose first that $G \wr H$ is cyclic subgroup separable. Since subgroups of cyclic subgroup separable groups are themselves cyclic subgroup separable, it follows that $G$ and $H$ are cyclic subgroup separable. Moreover, as $G \wr H$ is cyclic subgroup separable, it is residually finite, and therefore either $G$ is abelian or $H$ is finite by \cite[Theorem 3.2]{gruenberg}.

Conversely, suppose that $G$ and $H$ are cyclic subgroup separable, and either $G$ is abelian or $H$ is finite. If $G$ is abelian, then $G \wr H$ is cyclic subgroup separable by Proposition~\ref{p:AGcss}. On the other hand, if $H$ is finite, then $G \wr H$ is cyclic subgroup separable by Lemma~\ref{l:fincss}.
\end{proof}

\section{The counterexamples} \label{s:counterex}

In this section, we give examples of finitely generated groups that are residually finite but not equationally Noetherian, or vice versa. In particular, we prove Proposition~\ref{p:RF=/>EN} and Theorem~\ref{t:EN=/>RF}.

\subsection{Residually torsion-free nilpotent but not equationally Noetherian} \label{ss:RTFNnotEN}

We first give an example of a subgroup $G_1 \leq \prod_{r=1}^\infty UT_r(\mathbb{Z})$ in Proposition~\ref{p:RF=/>EN}\ref{i:pRF-ut}. Note that since the subgroup $UT_r(\mathbb{Z}) \leq GL_r(\mathbb{Z})$ can be easily seen to be torsion-free and $(r-1)$-step nilpotent, it follows that such a group $G_1$ is residually torsion-free nilpotent. Since $G_1$ is also finitely generated, it is thus residually finitely generated torsion-free nilpotent. But a finitely generated torsion-free nilpotent group is residually a finite $p$-group for any fixed prime $p$ \cite[Theorem 2.1 (i)]{gruenberg}, and so such a group $G_1$ must be residually a finite $p$-group for any fixed prime $p$. Of course, this also implies that $G_1$ is residually finite.

Thus, in order to prove Proposition~\ref{p:RF=/>EN}\ref{i:pRF-ut}, we only need to construct a finitely generated subgroup of $UT_r(\mathbb{Z})$ that is not equationally Noetherian. The example we use was inspired by the argument of A.~Bier in \cite{bier}.

Given an integer $r \geq 1$, we will write $I_r$ for the $r \times r$ identity matrix, and $E^{(i,j)}_r$ for the matrix whose $(i,j)$-th entry is $1$ and all other entries are $0$, where $1 \leq i < j \leq r$. It is easy to verify that
\begin{equation} \label{e:EE}
E^{(i,j)}_r E^{(k,\ell)}_r = \begin{cases} E^{(i,\ell)}_r & \text{if } j = k, \\ 0 & \text{otherwise}. \end{cases}
\end{equation}
We will write $I$ for $I_r$ and $E^{(i,j)}$ for $E^{(i,j)}_r$ if the value of $r$ is clear. Notice that we have $I_r+E_r \in UT_r(\mathbb{Z})$, where $E_r$ is any $\mathbb{Z}$-linear combination of the $E^{(i,j)}_r$, $1 \leq i < j \leq r$.

\begin{lem} \label{l:UT}
Let $r \geq 2$, and let $B = I + \sum_{i=1}^{r-1} E^{(i,i+1)} \in UT_r(\mathbb{Z})$. Then:
\begin{enumerate}[label=\textup{(\roman*)}]
\item \label{i:lUT1} $[I+E^{(1,2)},\overbrace{B,\ldots,B}^n] = I+E^{(1,n+2)}$ for $0 \leq n \leq r-2$, and $[I+E^{(1,2)},\overbrace{B,\ldots,B}^n] = I$ for $n \geq r-1$;
\item \label{i:lUT2} for any $n \geq 0$, $[I+E^{(1,2)},\overbrace{B,\ldots,B}^n,I+E^{(r-1,r)}] = I$ if and only if $n \neq r-3$.
\end{enumerate}
\end{lem}

\begin{proof} ~
\begin{enumerate}[label=(\roman*)]
\item We will prove the first statement by induction on $n$. The base case, $n = 0$, is trivial.

Let $n \geq 1$. By the induction hypothesis, it is enough to show that $[I+E^{(1,n+1)},B] = I+E^{(1,n+2)}$, which is equivalent to $(I+E^{(1,n+1)})B = B(I+E^{(1,n+1)})(I+E^{(1,n+2)})$. But it is easy to check, using \eqref{e:EE}, that both sides are equal to $B+E^{(1,n+1)}+E^{(1,n+2)}$, as required.

It is clearly enough to prove the second statement for $n = r-1$. Therefore, as a consequence of the first statement, it suffices to show that $I+E^{(1,r)}$ commutes with $B$. But it is easy to check, again using \eqref{e:EE}, that both $(I+E^{(1,r)})B$ and $B(I+E^{(1,r)})$ are equal to $B+E^{(1,r)}$.

\item If $n \geq r-1$, then the result follows trivially from part \ref{i:lUT1}. Thus, we may assume that $0 \leq n \leq r-2$. By part \ref{i:lUT1}, it is enough to show that $I+E^{(1,n+2)}$ and $I+E^{(r-1,r)}$ commute if and only if $n \neq r-3$, but this follows easily from \eqref{e:EE}. Indeed, if $n = r-3$ then $(I+E^{(1,n+2)})(I+E^{(r-1,r)}) = I+E^{(1,r-1)}+E^{(r-1,r)}+E^{(1,r)}$ whereas $(I+E^{(r-1,r)})(I+E^{(1,n+2)}) = I+E^{(1,r-1)}+E^{(r-1,r)}$. On the other hand, if $n \neq r-3$ then $(I+E^{(1,n+2)})(I+E^{(r-1,r)}) = I+E^{(1,n+2)}+E^{(r-1,r)} = (I+E^{(r-1,r)})(I+E^{(1,n+2)})$. \qedhere
\end{enumerate}
\end{proof}

We now construct the group $G_1$ as follows. Let $\widehat{G} = \prod_{k=1}^\infty UT_{2^k}(\mathbb{Z})$. Let
\[
\mathbf{A} = \left( I_{2^k}+E^{(1,2)}_{2^k} \right)_{k=1}^\infty, \quad \mathbf{B} = \left( I_{2^k}+\sum_{i=1}^{2^k-1} E^{(i,i+1)}_{2^k} \right)_{k=1}^\infty \quad \text{and} \quad \mathbf{C} = \left( I_{2^k}+E^{(2^k-1,2^k)}_{2^k} \right)_{k=1}^\infty,
\]
and let $G_1 = \langle \mathbf{A}, \mathbf{B}, \mathbf{C} \rangle < \widehat{G}$. 

\begin{proof}[Proof of Proposition~\ref{p:RF=/>EN}\ref{i:pRF-ut}]
As $G_1$ is finitely generated by definition, it is enough to show that it is not equationally Noetherian. Consider $S = \{ s_n(X_1,X_2,X_3) \mid n \geq 0 \} \subset F_3(X_1,X_2,X_3)$, where $s_n(X_1,X_2,X_3) = [X_1, \overbrace{X_2,\ldots,X_2}^n, X_3]$. For each $m \geq 0$, let $\mathbf{A}_m = [\mathbf{A},\overbrace{\mathbf{B},\ldots,\mathbf{B}}^m] \in G_1$. It then follows easily from Lemma~\ref{l:UT}\ref{i:lUT2} that $s_n(\mathbf{A}_m,\mathbf{B},\mathbf{C}) = 1$ if and only if $n+m+3$ is not a power of $2$. In particular, we have $(\mathbf{A}_{2^k-2},\mathbf{B},\mathbf{C}) \in V_{G_1}(S_{2^k-2}) \setminus V_{G_1}(S)$ for all $k \geq 1$, where $S_N = \{ s_n \mid 0 \leq n \leq N \}$ for any $N \geq 0$. Therefore, $V_{G_1}(S_{2^k-2}) \neq V_{G_1}(S)$ for any $k \geq 1$, and so $V_{G_1}(\overline{S}) \neq V_{G_1}(S)$ for any finite subset $\overline{S} \subset S$. Thus $G_1$ is not equationally Noetherian, as required.
\end{proof}

\subsection{Conjugacy separable but not equationally Noetherian} \label{ss:CSnotEN}

Here we give a construction of groups that are conjugacy separable and cyclic subgroup separable, but not equationally Noetherian -- in particular, we prove Proposition~\ref{p:RF=/>EN}\ref{i:pRF-wp}. Our construction is exemplified by the group $C_2 \wr (C_2 \wr \mathbb{Z})$. Indeed, this group is cyclic subgroup separable by two applications of Theorem~\ref{t:css}, and thus conjugacy separable by two applications of \cite[Theorem 1]{remwp}, and the group $C_2 \wr \mathbb{Z}$ has an infinite locally finite subgroup $\bigoplus_{j = -\infty}^\infty C_2$.

To deal with equational Noetherianity, we use the following result.

\begin{lem} \label{l:GwrHnen}
Let $G$ and $H$ be non-trivial groups. Suppose that $H$ has an infinite locally finite subgroup $K$. Then $G \wr H$ is not strongly equationally Noetherian. Moreover, if $H$ has a finitely generated subgroup $\overline{H}$ such that $K \cap \overline{H}$ is infinite, then $G \wr H$ is not equationally Noetherian.
\end{lem}

\begin{proof}
Since $K$ is infinite and locally finite, there exists an infinite strictly ascending series $\{1\} = K_0 \lneq K_1 \lneq K_2 \lneq \cdots$ of finite subgroups $K_n < K$. For each $n \geq 0$, let $h_n \in K_{n+1} \setminus K_n$ be an arbitrary element. Viewing $H$ as a subgroup of $G \wr H$, consider the system of equations
\[
S = \{ [X,h_n] \mid n \geq 0 \} \subset (G \wr H) * F_1(X).
\]
If $G \wr H$ was strongly equationally Noetherian, then clearly $V_{G \wr H}(S) = V_{G \wr H}(S_N)$ for some $N \geq 0$, where
\[
S_N = \{ [X,h_n] \mid 0 \leq n < N \} \subset S.
\]
We claim that this is not the case.

Let $g \in G$ be an arbitrary non-identity element. Given a finite subset $A \subset H$, we define $g^{(A)} := \prod_{a \in A} a^{-1}ga \in G \wr H$; note that, as $a^{-1}ga$ and $b^{-1}gb$ commute for any distinct $a,b \in H$, this product is uniquely defined, i.e.\ the order of multiplication does not matter. It is then clear that $g^{(A)} \neq g^{(B)}$ whenever $A \neq B$, and that $h^{-1}g^{(A)}h = g^{(Ah)}$ for any $h \in H$. In particular, if $A$ is a subgroup of $H$, then (for any $h \in H$) $[g^{(A)},h] = 1$ if and only if $h \in A$. But this shows that, for any $N \geq 0$, $g^{(K_N)} \in V_{G \wr H}(S_N)$, but $g^{(K_N)} \notin V_{G \wr H}(S)$ as $[g^{(K_N)},h_N] \neq 1$. Thus $V_{G \wr H}(S) \neq V_{G \wr H}(S_N)$ for any $N \geq 0$, as claimed. Hence $G \wr H$ is not strongly equationally Noetherian.

Now suppose that there exists a finitely generated subgroup $\overline{H} = \langle \overline{h}_1, \ldots, \overline{h}_k \rangle \leq H$ such that $K \cap \overline{H}$ is infinite. Then, after replacing $K$ with $K \cap \overline{H}$ if necessary, we may assume that $K \subseteq \overline{H}$. 
As $\overline{H}$ has a finite generating set $\{ \overline{h}_1, \ldots, \overline{h}_k \}$, for each $n \geq 0$ there exists a word $\varphi_n(Y_1,\ldots,Y_k) \in F_k(Y_1,\ldots,Y_k)$ such that $\varphi_n(\overline{h}_1, \ldots, \overline{h}_k) = h_n$.
Consider the system of equations
\[
S' = \{ [X,\varphi_n(Y_1,\ldots,Y_k)] \mid n \geq 0 \} \subset F_{k+1}(X,Y_1,\ldots,Y_k).
\]
Then, as before, it follows that $(g^{(K_N)},\overline{h}_1,\ldots,\overline{h}_k) \in (G \wr H)^{k+1}$ is an element of $V_{G \wr H}(S_N') \setminus V_{G \wr H}(S')$, where $S_N' = \{ [X,\varphi_n(Y_1,\ldots,Y_k)] \mid 0 \leq n < N \} \subset S'$. Thus $G \wr H$ is not equationally Noetherian, as required.
\end{proof}

\begin{proof}[Proof of Proposition~\ref{p:RF=/>EN}\ref{i:pRF-wp}]
As $G$ is non-trivial and $H$ is a finitely generated group with an infinite locally finite subgroup, the first part follows from Lemma~\ref{l:GwrHnen}.

For the second part, suppose in addition that $G$ is finitely generated abelian and $H$ is cyclic subgroup separable and conjugacy separable. Then $G$ is also conjugacy separable, and so $G \wr H$ is conjugacy separable by a result of V.~P.~Remeslennikov \cite[Theorem 1]{remwp}. Moreover, $G$ is cyclic subgroup separable, and so $G \wr H$ is cyclic subgroup separable by Theorem~\ref{t:css}.
\end{proof}

\subsection{Equationally Noetherian but not residually finite} \label{ss:ENnotRF}

We now construct finitely presented groups that are equationally Noetherian but not residually finite: namely, we prove Theorem~\ref{t:EN=/>RF}. For the second part of this result, let $H$ be the inverse image of $Sp_{2n}(\mathbb{Z})$ under the universal covering map $\widetilde{Sp_{2n}(\mathbb{R})} \to Sp_{2n}(\mathbb{R})$. Then $H$ is not residually finite by a result of P.~Deligne \cite{deligne}.

Moreover, as $Sp_{2n}(\mathbb{R})$ is a connected semisimple algebraic group defined over $\mathbb{Q}$ and as its subgroup $G = Sp_{2n}(\mathbb{Z})$ is arithmetic, a result of M.~S.~Raghunathan \cite[Corollary 2]{raghu} says that $G$ is finitely presented. We may deduce from this that $H$ is finitely presented. Indeed, $H$ is an extension of its central subgroup $Z \cong \pi_1(Sp_{2n}(\mathbb{R}) \cong \mathbb{Z}$ by $G$. Thus, if $G = \langle \mathcal{X} \mid \mathcal{R} \rangle$ is a finite presentation for $G$, then it is easy to see that $H \cong \langle \mathcal{X} \cup \{z\} \mid \mathcal{R}_1 \cup \mathcal{R}_2 \rangle$, where $\mathcal{R}_1 = \{ rz^{-n(r)} \mid r \in \mathcal{R} \}$ for some integers $\{ n(r) \mid r \in \mathcal{R} \}$, and $\mathcal{R}_2 = \{ [x,z] \mid x \in \mathcal{X} \}$.

It is therefore left to show that groups $H$ in Theorem~\ref{t:EN=/>RF} are strongly equationally Noetherian. We do this below.

\begin{proof}[Proof of Theorem~\ref{t:EN=/>RF}]
Note first that $\ker(p)$ is central in $H$. Indeed, let $z \in \ker(p)$. If $h \in H$ then, as $H$ is a connected Lie group, there exists a continuous path $\gamma_h: [0,1] \to H$ from $1$ to $h$. Then the path $\beta_h: [0,1] \to H, t \mapsto [z,\gamma_h(t)] = z^{-1} \cdot z^{\gamma_h(t)}$ is also continuous, and since $\ker(p)$ is normal in $H$ we have $\beta_h([0,1]) \subseteq \ker(p)$. But since $p$ is a covering map, $\ker(p)$ is discrete in $H$, and so $\beta_h$ must be constant. Thus $[z,h] = \beta_h(1) = \beta_h(0) = [z,1] = 1$ for all $h \in H$ and so $z \in Z(H)$, as claimed.

We may embed $GL_m(\mathbb{R})$, and hence $G$, as a Zariski-closed subset of $\mathbb{R}^{m^2+1}$. Thus $G^n$ is a Zariski-closed subset of $\mathbb{R}^{n(m^2+1)}$. 

Now let $\widetilde{S} \leq F_n * H$ be a system of exponent sum zero equations over $H$: it is enough to consider such systems of equations by Lemma~\ref{l:+0}. Let $S = \Phi(\widetilde{S})$, where $\Phi: F_n * H \to F_n * G$ is the homomorphism defined by $\Phi(f) = f$ for $f \in F_n$ and $\Phi(g) = p(g)$ for $g \in H$. As $G$ is linear, and so strongly equationally Noetherian, it follows that $V := V_G(S) \subseteq G^n$ is equal to $V_G(S_0)$ for some finite subset $S_0 \subseteq S$. Now for each $s \in S_0$, choose an element $\widetilde{s} \in \Phi^{-1}(s) \cap \widetilde{S}$, and define a finite subset $\widetilde{S}_0 = \{ \widetilde{s} \mid s \in S_0 \} \subseteq \widetilde{S}$.

Note that for any $s \in F_n * G$ and any $A_1,\ldots,A_n \in G \leq GL_m(\mathbb{R})$, entries of the matrix $s(A_1,\ldots,A_n)$ can be expressed as polynomials (with real coefficients) of the entries of matrices $A_1,\ldots,A_n$. It follows that $V$ is a Zariski-closed subset of $G^n$, and so it is an affine algebraic variety over $\mathbb{R}$. Therefore, by \cite[Theorem 4.1]{dekn}, $V$ has finitely many path components: $V_1,\ldots,V_r$, say.

Now as $K = \ker(p)$ is central in $H$, it follows that given any exponent sum zero equation $s \in F_n * H$ and any $g_1,\ldots,g_n \in H$, we have $s(g_1,\ldots,g_n) = s(k_1g_1,\ldots,k_ng_n)$ for any $k_1,\ldots,k_n \in K$. Thus $s(g_1,\ldots,g_n)$ depends only on the element $p^{(n)}(g_1,\ldots,g_n) := (p(g_1),\ldots,p(g_n)) \in G^n$, and so the map $s: H^n \to H$ factors as $s = \widehat{s} \circ p^{(n)}$ for a map $\widehat{s}: G^n \to H$ and the map $p^{(n)}: H^n \to G^n$. Moreover, it follows from the construction that we have $V = \{ (g_1,\ldots,g_n) \in G^n \mid \widehat{s}(g_1,\ldots,g_n) \in K \text{ for all } s \in \widetilde{S} \}$, and so each $s \in \widetilde{S}$ defines a function $\widehat{s}|_V: V \to K$.

But as $s: H^n \to H$ is clearly continuous for every $s \in \widetilde{S}$ and as $p^{(n)}$ is a covering map, $\widehat{s}$ is also continuous and hence so is $\widehat{s}|_V$. As $K$ is discrete (by the definition of a covering map), it follows that $\widehat{s}|_V$ is constant on each of the path components $V_i$ of $V$, for $1 \leq i \leq r$. In particular, for each $s \in \widetilde{S}$ and each $i \in \{ 1,\ldots,r \}$, there exists $k_{s,i} \in K$ such that $\widehat{s}(g_1,\ldots,g_n) = k_{s,i}$ for all $(g_1,\ldots,g_n) \in V_i$.

It follows that we have
\[
V_{H}(\widetilde{S}) = \bigcup \left\{ (p^{(n)})^{-1}(V_i) \:\middle|\: k_{s,i} = 1 \text{ for all } s \in \widetilde{S} \right\} = \bigcup_{i \in \mathcal{I}} (p^{(n)})^{-1}(V_i)
\]
for some subset $\mathcal{I} \subseteq \{1,\ldots,r\}$. For each $i \in \{1,\ldots,r\} \setminus \mathcal{I}$, let $s_i \in \widetilde{S}$ be an element such that $k_{s_i,i} \neq 1$, and define a finite subset $\widetilde{S}_1 = \{ s_i \mid 1 \leq i \leq r, i \notin \mathcal{I} \}$ of $\widetilde{S}$. It then follows from the construction that $V_{H}(\widetilde{S}_0) \subseteq (p^{(n)})^{-1}(V)$ and so $V_{H}(\widetilde{S}_0 \cup \widetilde{S}_1) = \bigcup_{i \in \mathcal{I}} (p^{(n)})^{-1}(V_i) = V_{H}(\widetilde{S})$. Thus, by Lemma~\ref{l:+0}, $H$ is strongly equationally Noetherian, as required.
\end{proof}

\section{Graphs of groups} \label{s:comb}

In this section, we deal with a finite graph of groups. Here we recall its definition and record our notation.

\begin{defn} \label{d:gg}
A \emph{finite graph of groups} $\mathcal{G}$ consists of the following data:
\begin{enumerate}[label=(\alph*)]
\item \label{d:gg-v} A finite set $V(\mathcal{G})$, the \emph{vertex set} of $\mathcal{G}$. 
\item \label{d:gg-e} A finite set $E(\mathcal{G})$, the \emph{edge set} of $\mathcal{G}$, together with a fixed point free involution $E(\mathcal{G}) \to E(\mathcal{G}), e \mapsto \overline{e}$.
\item \label{d:gg-it} A function $i: E(\mathcal{G}) \to V(\mathcal{G})$, the \emph{initial vertex function}, such that for every non-empty proper subset $A \subset V(\mathcal{G})$ there exists $e \in E(\mathcal{G})$ with $i(e) \in A$ and $i(\overline{e}) \notin A$.
\item \label{d:gg-gv} A group $G_v$ for each vertex $v \in V(\mathcal{G})$: the \emph{vertex groups}.
\item \label{d:gg-ge} A group $G_e = G_{\overline{e}}$ for each pair of edges $\{ e,\overline{e} \} \subseteq E(\mathcal{G})$: the \emph{edge groups}.
\item \label{d:gg-git} An injective group homomorphism $\iota_e: G_e \to G_{i(e)}$ for each $e \in E(\mathcal{G})$: the \emph{inclusion maps}.
\end{enumerate}
\end{defn}

We may use parts \ref{d:gg-v}--\ref{d:gg-it} to define a finite connected undirected graph (called the \emph{underlying graph} of $\mathcal{G}$) with vertex set $V(\mathcal{G})$ and an edge for each pair $e,\overline{e} \in E(\mathcal{G})$, viewed as a pair of a directed edge $e$ (from $i(e)$ to $i(\overline{e})$) and its inverse. Let $T \subseteq E(\mathcal{G})$ be a maximal tree in this graph, so that $e \in T$ if and only if $\overline{e} \in T$ for each $e \in E(\mathcal{G})$. Then we may define the \emph{fundamental group} $\pi_1(\mathcal{G},T)$ of $\mathcal{G}$ with respect to $T$ as the group generated by $\left( \bigsqcup_{v \in V(\mathcal{G})} G_v \right) \sqcup E(\mathcal{G})$ with the following relations:
\begin{enumerate}[label=(R\alph*)]
\item \label{d:gg-rv} the multiplication tables of $G_v$ for each $v \in V(\mathcal{G})$: that is, $(g)(h) = (gh)$ for $g,h \in G_v$;
\item \label{d:gg-rbs} the \emph{Bass--Serre relations}: $\iota_{\overline{e}}(g) = \overline{e}\iota_e(g)e$ for each $e \in E(\mathcal{G})$ and $g \in G_e$;
\item \label{d:gg-re} $e\overline{e} = 1$ for each $e \in E(\mathcal{G})$; and
\item \label{d:gg-rt} $e = 1$ for each $e \in T$.
\end{enumerate}

It is a standard fact in Bass--Serre theory that $\pi_1(\mathcal{G},T)$ is independent (up to isomomorphism) from the choice of $T$, and so we will write $\pi_1(\mathcal{G})$ for $\pi_1(\mathcal{G},T)$. It is also clear from the relations of types \ref{d:gg-rv} and \ref{d:gg-rt} that each element of $\pi_1(\mathcal{G})$ can be written as $g_0 e_1 g_1 \cdots e_n g_n$ with $g_j \in G_{v_j}$ for some $v_j \in V(\mathcal{G})$ and $e_j \in E(\mathcal{G})$, such that $v_n = v_0$, $v_{j-1} = i(e_j)$ and $v_j = i(\overline{e_j})$ for $1 \leq j \leq n$. The following result then states that a trivial element in $\pi_1(\mathcal{G})$ is `obviously' trivial.

\begin{thm}[Normal form theorem; see {\cite[Chapter I, Theorem 4.1]{dd}}] \label{t:ftbst}
Let $\mathcal{G}$ be a finite graph of groups, and let $g = g_0 e_1 g_1 \cdots e_n g_n \in \pi_1(\mathcal{G})$ with $g_j \in G_{v_j}$ for some $v_j \in V(\mathcal{G})$ and $e_j \in E(\mathcal{G})$, such that $v_n=v_0$, $v_{j-1} = i(e_j)$ and $v_j = i(\overline{e_j})$ for $1 \leq j \leq n$. If $g = 1$ in $\pi_1(\mathcal{G})$, then either $n = 0$ and $g_0 = 1$ in $G_{v_0}$, or $n \geq 2$ and there exists $j \in \{1,\ldots,n-1\}$ such that $e_{j+1} = \overline{e_j}$ and $g_j \in \iota_{\overline{e_j}}(G_{e_j})$.
\end{thm}

It follows, in particular, that $G_v$ is canonically a subgroup of $\pi_1(\mathcal{G})$ for each $v \in V(\mathcal{G})$. One of the main facts in Bass--Serre theory is that $\pi_1(\mathcal{G})$ acts simplicially, minimally and cocompactly (with quotient the underlying graph of $\mathcal{G}$) on a simplicial tree $\mathcal{T_G}$, the \emph{Bass--Serre tree}, such that vertex stabilisers and edge stabilisers under this action are precisely the conjugates of $G_v$, $v \in V(\mathcal{G})$ and $\iota_e(G_e)$, $e \in E(\mathcal{G})$, respectively. We refer the reader to \cite[Chapter I]{dd} for a comprehensive account of Bass--Serre theory.

In our proof of Theorem~\ref{t:comb} we will use the following strengthening of condition \ref{i:tcomb-fin1v}.

\begin{lem} \label{l:fin1v2v}
The condition \ref{i:tcomb-fin1v} in Theorem~\ref{t:comb} is equivalent to the following statement (for any finite graph of groups $\mathcal{G}$ and any homomorphisms $\overline\phi_e: G_{i(e)} \to G_{i(\overline{e})}$ extending $\phi_e = \iota_{\overline{e}} \circ \iota_e^{-1}$):
\begin{enumerate}[label=\textup{(\roman*')},start=2]
\item \label{i:tcomb-fin2v} for any pair of vertices $v,w \in V(\mathcal{G})$, there are only finitely many homomorphisms of the form $\overline\phi_{e_n} \circ \cdots \circ \overline\phi_{e_1}: G_v \to G_w$, where $e_1 \cdots e_n$ ranges over all paths in $\mathcal{G}$ from $v$ to $w$.
\end{enumerate}
\end{lem}

\begin{proof}
The implication \ref{i:tcomb-fin2v} $\Rightarrow$ \ref{i:tcomb-fin1v} is trivial. We thus only need to prove \ref{i:tcomb-fin1v} $\Rightarrow$ \ref{i:tcomb-fin2v}.

Suppose that \ref{i:tcomb-fin1v} holds, and let $v,w \in V(\mathcal{G})$. Let $m = |V(\mathcal{G})|$. Since $m < \infty$, there exist $D < \infty$ endomorphisms of the form $\overline\phi_{e_n} \circ \cdots \circ \overline\phi_{e_1}: G_u \to G_u$, where $e_1 \cdots e_n$ ranges over all closed paths in $\mathcal{G}$ based at $u$ (for some $u \in V(\mathcal{G})$); we will call an endomorphism of this form a \emph{closed path endomorphism}. Since $\mathcal{G}$ is finite, there also exist finitely many paths in $\mathcal{G}$ which do not visit the same vertex more than once, and so there are $k < \infty$ homomorphisms of the form $\overline\phi_{e_n} \circ \cdots \circ \overline\phi_{e_1}$ for such a path $e_1 \cdots e_n$; we will call an homomorphism of this form a \emph{simple path homomorphism}.

Now let $P$ be a path in $\mathcal{G}$ from $v$ to $w$. Express $P$ as $P = P_0 Q_1 P_1 \cdots Q_r P_r$ in such a way that each $Q_i$ is a closed path (based at some $u_i \in V(\mathcal{G})$), and so that $r+2\sum_{i=0}^r |P_i|$ is as small as possible; here $|Q|$ denotes the length of a path $Q$, i.e.\ its number of edges. Then $u_i \neq u_j$ for $i \neq j$: indeed, if we had $u_i = u_j$ for some $i < j$ then we could write $P = P_0 Q_1 \cdots P_{i-1} Q' P_j \cdots Q_r P_r$ for a closed path $Q'$, contradicting the minimality of $r+2\sum_{i=0}^r |P_i|$. It follows that $r \leq m$. Similarly, we can see that each $P_i$ does not visit any vertex more than once: otherwise we would have $P_i = P_i' Q_i' P_i''$ for a closed path $Q_i'$ with $|Q_i'| \geq 1$, and so we could write $P = P_0 Q_1 \cdots Q_i P_i' Q_i' P_i'' Q_{i+1} \cdots Q_r P_r$, again contradicting the minimality of $r+2\sum_{i=0}^r |P_i|$.

It follows that for any path $e_1 \cdots e_n$ in $\mathcal{G}$ from $v$ to $w$, the homomorphism $\overline\phi_{e_n} \circ \cdots \circ \overline\phi_{e_1}: G_v \to G_w$ is of the form $\psi_r \circ \xi_r \circ \cdots \circ \psi_1 \circ \xi_1 \circ \psi_0$, where $r \leq m$, each $\psi_i$ is a simple path homomorphism, and each $\xi_i$ is a closed path endomorphism. As there are $D < \infty$ choices for each $\xi_i$ and $k < \infty$ choices for each $\psi_i$, it follows that there are at most $k^{m+1}D^m < \infty$ choices for the homomorphism $\overline\phi_{e_n} \circ \cdots \circ \overline\phi_{e_1}$. This proves \ref{i:tcomb-fin2v}.
\end{proof}

The main aim of this section is to prove Theorem~\ref{t:comb}. We start by proving the implication ($\Leftarrow$) of part \ref{i:tcomb-RF} in Section~\ref{ss:combRF}. In Section~\ref{ss:qalg}, we introduce quasi-algebraic subsets of groups and discuss ways to construct them. In Section~\ref{ss:acyl} we discuss acylindrical hyperbolicity of the action $\pi_1(\mathcal{G}) \curvearrowright \mathcal{T_G}$, and in Section~\ref{ss:combEN} we prove the implication ($\Leftarrow$) of part \ref{i:tcomb-EN}. In Section~\ref{ss:comb-rem} we prove the implications ($\Rightarrow$) in both parts and make a few remarks.

\subsection{Residually finite vertex groups} \label{ss:combRF}

We first prove the implication ($\Leftarrow$) in Theorem~\ref{t:comb}\ref{i:tcomb-RF}.

\begin{prop} \label{p:comb-RF-<=}
Let $\mathcal{G}$ be a finite graph of groups. Suppose that
\begin{enumerate}[label={\textup{(\roman*)}}]
\item \label{i:pcomb-rf-ext} for each $e \in E(\mathcal{G})$, the isomorphism $\phi_e = \iota_{\overline{e}} \circ \iota_e^{-1}: \iota_e(G_e) \to \iota_{\overline{e}}(G_e)$ extends to a homomorphism $\overline\phi_e: G_{i(e)} \to G_{i(\overline{e})}$;
\item \label{i:pcomb-rf-fin1v} for each $v \in V(\mathcal{G})$, there are only finitely many endomorphisms of the form $\overline\phi_{e_n} \circ \cdots \circ \overline\phi_{e_1}: G_v \to G_v$, where $e_1 \cdots e_n$ ranges over all closed paths in $\mathcal{G}$ starting and ending at $v$;
\item \label{i:pcomb-rf-en} for each $v \in V(\mathcal{G})$, the group $G_v$ is residually finite; and
\item \label{i:pcomb-rf-qa} for each $e \in E(\mathcal{G})$, the subgroup $\iota_e(G_e) \leq G_{i(e)}$ is separable.
\end{enumerate}
Then $\pi_1(\mathcal{G})$ is residually finite.
\end{prop}

\begin{proof}
We say a collection of normal subgroups $\{ N_v \unlhd G_v \}_{v \in V(\mathcal{G})}$ is \emph{compatible} if $|G_v/N_v| < \infty$ for all $v \in V(\mathcal{G})$ and if $\iota_e^{-1}(N_{i(e)}) = \iota_{\overline{e}}^{-1}(N_{i(\overline{e})})$ for all $e \in E(\mathcal{G})$. Given a compatible collection of subgroups $N_v \unlhd G_v$, we may construct a graph of groups $\widehat{\mathcal{G}}$ with the same underlying graph as $\mathcal{G}$, with vertex groups $\widehat{G}_v = G_v/N_v$, edge groups $\widehat{G}_e = G_e/\iota_e^{-1}(N_{i(e)})$ and the inclusion maps $\widehat\iota_e$ induced by $\iota_e$. It is easy to check that in this case the quotient maps $G_v \to \widehat{G}_v$ extend to a surjective homomorpism $p_{\{N_v\}}: \pi_1(\mathcal{G}) \to \pi_1(\widehat{\mathcal{G}})$. Moreover, as $\widehat{\mathcal{G}}$ is a finite graph of groups with finite vertex groups, it follows that $\pi_1(\widehat{\mathcal{G}})$ is virtually free and hence residually finite. Thus, in order to show that $\pi_1(\mathcal{G})$ is residually finite, it is enough to show that for each non-trivial element $g \in \pi_1(\mathcal{G})$, there exists a compatible collection $\{ N_v \unlhd G_v \}$ such that $p_{\{N_v\}}(g) \neq 1$.

By Theorem~\ref{t:ftbst}, it is therefore enough:
\begin{enumerate}[label=(\alph*)]
\item \label{i:pcrf-1} for each $w \in V(\mathcal{G})$ and each $g \in G_w \setminus \{1\}$, to find a compatible collection $\{ N_v \unlhd G_v \}$ such that $g \notin N_w$; and
\item \label{i:pcrf-2} for each $e \in E(\mathcal{G})$ and each $g \in G_{i(e)} \setminus \iota_e(G_e)$, to find a compatible collection $\{ N_v \unlhd G_v \}$ such that $g \notin N_{i(e)}\iota_e(G_e)$.
\end{enumerate}
Now fix a vertex $w \in V(\mathcal{G})$, a separable subgroup $A < G_w$ and an element $g \in G_w \setminus A$. Then, by assumptions \ref{i:pcomb-rf-en} and \ref{i:pcomb-rf-qa} of the Proposition, both \ref{i:pcrf-1} and \ref{i:pcrf-2} will follow if we can find a compatible collection $\{ N_v \unlhd G_v \}$ such that $g \notin N_wA$.

To find such a collection, let $N \lhd G_w$ be a finite-index normal subgroup such that $g \notin NA$: such a subgroup exists as $A$ is separable. For each $v \in V(\mathcal{G})$, let
\[
N_v = \bigcap \left\{ (\overline\phi_{e_n} \circ \cdots \circ \overline\phi_{e_1})^{-1}(N) \:\middle|\: e_1 \cdots e_n \text{ is a path in $\mathcal{G}$ from $w$ to $v$} \right\},
\]
where the homomorphisms $\overline\phi_{e_j}: G_{i(e_j)} \to G_{i(\overline{e_j})}$ are as in assumption \ref{i:pcomb-rf-ext}. Note that there is an empty path from $w$ to $w$, and so $N_w \subseteq N$, implying that $g \notin N_wA$. Moreover, it follows from assumption \ref{i:pcomb-rf-fin1v} and Lemma~\ref{l:fin1v2v} that $N_v$ is an intersection of \emph{finitely many} finite-index normal subgroups of $G_v$, and so is itself a finite-index normal subgroup. Finally, given an edge $e \in E(\mathcal{G})$, if $e_1 \cdots e_n$ is a path from $i(\overline{e})$ to $v$ then $ee_1 \cdots e_n$ is a path from $i(e)$ to $v$, which implies that $N_{i(e)} \subseteq (\overline\phi_e)^{-1}(N_{i(\overline{e})})$ and so $\iota_e^{-1}(N_{i(e)}) \subseteq \iota_{\overline{e}}^{-1}(N_{i(\overline{e})})$; similarly, $\iota_e^{-1}(N_{i(e)}) \supseteq \iota_{\overline{e}}^{-1}(N_{i(\overline{e})})$. Hence the collection $\{N_v \unlhd G_v\}_{v \in V(\mathcal{G})}$ is compatible, thus completing the proof.
\end{proof}

\subsection{Quasi-algebraic subsets} \label{ss:qalg}

We now discuss the collection of quasi-algebraic subsets in a group. Recall that given a group $G$, we say a subset $A \subseteq G$ is \emph{quasi-algebraic} in $G$ if given any $n \geq 1$ and any $S \subseteq F_n$ we have $V_{G,A}(S) = V_{G,A}(S_0)$ for some finite subset $S_0 \subseteq S$, where we define
\[
V_{G,A}(T) = \{ (g_1,\ldots,g_n) \in G^n \mid t(g_1,\ldots,g_n) \in A \text{ for all } t \in T \}
\]
for any $T \subseteq F_n$. Thus $\{1\}$ is quasi-algebraic in $G$ if and only if $G$ is equationally Noetherian, and more generally a normal subgroup $N \unlhd G$ is quasi-algebraic in $G$ if and only if $G/N$ is equationally Noetherian. We thus view quasi-algebraic subsets in the context of equationally Noetherian groups as an analogue of separable subsets in the context of residually finite groups.

\begin{rmk}
In spite of this analogy, the collection of quasi-algebraic subsets is in general not the collection of closed sets in a topology, as this collection need not be closed under arbitrary intersections. For instance, let $\mathcal{G}$ be a family of equationally Noetherian groups that is not an equationally Noetherian family (such families are plentiful by Corollary~\ref{c:ENfam}, for example), and let $G = \bigoplus_{H \in \mathcal{G}} H$. Then the group $G$ is not equationally Noetherian (see, for instance, \cite[Lemma 3.9]{gh}), and so $\{1\}$ is not quasi-algebraic in $G$. However, we have $\{1\} = \bigcap_{H \in \mathcal{G}} \ker(p_H)$ and $\ker(p_H)$ is quasi-algebraic for each $H \in \mathcal{G}$, where $p_H: G \to H$ is the canonical projection for any $H \in \mathcal{G}$.
\end{rmk}

The following result says, however, that many separable subsets of an equationally Noetherian residually finite group are also quasi-algebraic. Note that an analogous result holds in the context of separable subsets, and so this will not help us to construct `new' quasi-algebraic subsets that are not separable.

\begin{prop}
The following subsets of a group $G$ are quasi-algebraic:
\begin{enumerate}[label={\textup{(\roman*)}}]
\item \label{i:qa-fi} cosets of finite index subgroups of $G$;
\item \label{i:qa-ui} finite unions and finite intersections of quasi-algebraic subsets of $G$;
\item \label{i:qa-hom} $\varphi^{-1}(A)$ for a group homomorphism $\varphi: G \to H$ and a quasi-algebraic subset $A \subseteq H$;
\item \label{i:qa-eq} $\{ g \mid s(g,g_2,\ldots,g_n) \in A \}$ for an equation $s \in F_n$ and a quasi-algebraic subset $A \subseteq G$;
\item \label{i:qa-ret} if $G$ is equationally Noetherian, retracts of $G$.
\end{enumerate}
In particular, quasi-algebraic subsets include inverses and translates of quasi-algebraic subsets.
\end{prop}

\begin{proof} We prove the parts in order.
\begin{enumerate}[label=(\roman*)]
\item If $H$ is a finite-index subgroup of $G$, then $N = \bigcap_{g \in G} H^g$ is a normal finite-index subgroup of $G$ contained in $H$. Since any subset of the finite group $G/N$ is clearly quasi-algebraic, it follows that any union of cosets of $N$ in $G$ is quasi-algebraic as well. In particular, any left or right coset of $H$ in $G$ is quasi-algebraic, as required.
\item Let $A_1,\ldots,A_m \subseteq G$ be a collection of quasi-algebraic subsets, and let $S \subseteq F_n$. Then, as each $A_j$ is quasi-algebraic, there exist finite subsets $S_j \subseteq S$ such that $V_{G,A_j}(S_j) = V_{G,A_j}(S)$ for each $j$. But now if $A = \bigcup_{j=1}^m A_j$, then we have $V_{G,A}(S) = \bigcup_{j=1}^m V_{G,A_j}(S) = \bigcup_{j=1}^m V_{G,A_j}(S_j) = \bigcup_{j=1}^m V_{G,A_j}(S_1 \cup \cdots \cup S_m) = V_{G,A}(S_1 \cup \cdots \cup S_m)$, and similarly $V_{G,A'}(S) = V_{G,A'}(S_1 \cup \cdots \cup S_m)$ for $A' = \bigcap_{j=1}^m A_j$. Thus $A$ and $A'$ are quasi-algebraic, as required.
\item Given any $T \subseteq F_n$, we have
\begin{align*}
V_{G,\varphi^{-1}(A)}(T) &= \{ (g_1,\ldots,g_n) \in G^n \mid \varphi(t(g_1,\ldots,g_n)) \in A \text{ for all } t \in T \} \\ &= \{ (g_1,\ldots,g_n) \in G^n \mid t(\varphi(g_1),\ldots,\varphi(g_n)) \in A \text{ for all } t \in T \} \\ &= \bigcup \{ \varphi^{-1}(h_1) \times \cdots \times \varphi^{-1}(h_n) \mid (h_1,\ldots,h_n) \in V_{H,A}(T) \}.
\end{align*}
Thus, since for any $S \subseteq F_n$ we have $V_{H,A}(S_0) = V_{H,A}(S)$ for some finite subset $S_0 \subseteq S$ (as $A \subseteq H$ is quasi-algebraic), it follows that we also have $V_{G,\varphi^{-1}(A)}(S_0) = V_{G,\varphi^{-1}(A)}(S)$, and so $\varphi^{-1}(A)$ is quasi-algebraic in $G$, as required.
\item Let $B = \{ g \mid s(g,g_2,\ldots,g_n) \in A \} \subseteq G$. Given an equation $t \in F_m(X_1,\ldots,X_m)$, we can define an equation $\widehat{t} \in F_{m+n-1}(X_1,\ldots,X_{m+n-1})$ by setting \[ \widehat{t}(X_1,\ldots,X_{m+n-1}) = s(t(X_1,\ldots,X_m),X_{m+1},\ldots,X_{m+n-1}). \]
Now if for any $T \subseteq F_m$ we set $\widehat{T} = \{ \widehat{t} \mid t \in T \} \subseteq F_{m+n-1}$, then we have
\begin{align*}
V_{G,B}(T) &= \{ (h_1,\ldots,h_m) \in G^m \mid s(t(h_1,\ldots,h_m),g_2,\ldots,g_n) \in A \text{ for all } t \in T \} \\ &= \{ (h_1,\ldots,h_m) \in G^m \mid (h_1,\ldots,h_m,g_2,\ldots,g_n) \in V_{G,A}(\widehat{T}) \}.
\end{align*}
In particular, since for any $S \subseteq F_m$ we have $V_{G,A}(\widehat{S_0}) = V_{G,A}(\widehat{S})$ for some finite subset $S_0 \subseteq S$ (as $A \subseteq G$ is quasi-algebraic), it follows that we also have $V_{G,B}(S_0) = V_{G,B}(S)$, and so $B$ is quasi-algebraic in $G$, as required.
\item Suppose now that $G$ is equationally Noetherian and $H \leq G$ is a retract of $G$: that is, there exists a homomorphism $\varphi: G \to H$ such that $\varphi(h) = h$ for all $h \in H$. Hence, for any $g \in G$, we have $g^{-1}\varphi(g) = 1$ if and only if $g \in H$. Given an equation $t \in F_n(X_1,\ldots,X_n)$, we can define an equation $\widehat{t} \in F_{2n}(X_1,\ldots,X_{2n})$ by setting \[ \widehat{t}(X_1,\ldots,X_{2n}) = t(X_1,\ldots,X_n)^{-1} t(X_{n+1},\ldots,X_{2n}). \]
Now if for any $T \subseteq F_n$ we set $\widehat{T} = \{ \widehat{t} \mid t \in T \} \subseteq F_{2n}$, then we have
\begin{align*}
V_{G,H}(T) &= \{ (g_1,\ldots,g_n) \in G^n \mid t(g_1,\ldots,g_n) \in H \text{ for all } t \in T \} \\ &= \{ (g_1,\ldots,g_n) \in G^n \mid t(g_1,\ldots,g_n)^{-1} t(\varphi(g_1),\ldots,\varphi(g_n)) = 1 \text{ for all } t \in T \} \\ &= \{ (g_1,\ldots,g_n) \in G^n \mid (g_1,\ldots,g_n,\varphi(g_1),\ldots,\varphi(g_n)) \in V_G(\widehat{T}) \}.
\end{align*}
In particular, since for any $S \subseteq F_n$ we have $V_G(\widehat{S_0}) = V_G(\widehat{S})$ for some finite subset $S_0 \subseteq S$ (as $G$ is equationally Noetherian), it follows that we also have $V_{G,H}(S_0) = V_{G,H}(S)$, and so $H$ is quasi-algebraic in $G$, as required.
\end{enumerate}
The last part follows from \ref{i:qa-eq}: namely, if $A \subseteq G$ is quasi-algebraic and $g \in G$, then $Ag$, $gA$ and $A^{-1}$ are quasi-algebraic (consider equations $X_1X_2,X_2X_1 \in F_2(X_1,X_2)$ and $X_1^{-1} \in F_1(X_1)$).
\end{proof}

In Section \ref{ss:combEN}, we use an alternative definition of quasi-algebraic subsets, inspired by the viewpoint of D.~Groves and M.~Hull \cite{gh}. For this, let $G$ be a group and $A \subseteq G$ a subset. Let $\omega$ be a non-principal ultrafilter (see Section~\ref{s:limit} for the relevant definitions). Given a finitely generated group $F$ and a sequence of homomorphisms $(\varphi_j: F \to G)_{j \in \mathbb{N}}$, we define
\[
K_A^{(\varphi_j)} = \{ f \in F \mid \varphi_j(f) \in A \text{ $\omega$-almost surely} \}.
\]
It is easy to see that if $A$ is a subgroup (respectively normal subgroup) of $G$, then $K_A^{(\varphi_j)}$ is a subgroup (respectively normal subgroup) of $F$.

The following result generalises a result of D.~Groves and M.~Hull \cite[Theorem 3.5]{gh}. Its proof is almost identical to the one given in \cite{gh}, and is therefore omitted.

\begin{thm} \label{t:qae}
Let $\omega$ be a non-principal ultrafilter. For a group $G$ and a subset $A \subseteq G$, the following are equivalent:
\begin{enumerate}[label={\textup{(\roman*)}}]
\item \label{i:tqae-n} $A$ is quasi-algebraic in $G$.
\item \label{i:tqae-as} For any finitely generated group $F$ and any sequence $(\varphi_j: F \to G)_{j \in \mathbb{N}}$ of homomorphisms, we $\omega$-almost surely have $K_A^{(\varphi_j)} \subseteq \varphi_j^{-1}(A)$.
\item \label{i:tqae-ex} For any finitely generated group $F$ and any sequence $(\varphi_j: F \to G)_{j \in \mathbb{N}}$ of homomorphisms, we have $K_A^{(\varphi_j)} \subseteq \varphi_j^{-1}(A)$ for some $j \in \mathbb{N}$. \qed
\end{enumerate}
\end{thm}

\subsection{Acylindrical hyperbolicity} \label{ss:acyl}

Recall that an action of a group $G$ on a geodesic metric space $X$ by isometries is said to be \emph{acylindrical} if for each $\varepsilon \geq 0$, there exist constants $D_\varepsilon,N_\varepsilon \geq 0$ such that for any $x,y \in X$ with $d_X(x,y) \geq D_\varepsilon$, the set $\mathcal{N}_{x,y,\varepsilon} = \{ g \in G \mid d_X(x,xg) \leq \varepsilon, d_X(y,yg) \leq \varepsilon \}$ has cardinality $\leq N_\varepsilon$. Such an action is said to be \emph{non-elementary} if $G$ is not virtually cyclic and some orbit under the action is unbounded. A group that acts non-elementarily acylindrically on a hyperbolic metric space is called \emph{acylindrically hyperbolic}.

We highlight one situation when the action of the fundamental group of a graph of groups on the associated Bass--Serre tree is acylindrical. In the following Lemma, we say a collection of subgroups $( H_\lambda \mid \lambda \in \Lambda )$ of a group $G$ is \emph{uniformly almost malnormal} if there exists a constant $D < \infty$ such that whenever $\lambda,\mu \in \Lambda$ and $g \in G$ satisfy $|H_\lambda \cap H_\mu^g| > D$, we have $\lambda = \mu$ and $g \in H_\lambda$.

The following result is well-known: see \cite{selaAH}, for instance. We include the proof here for completeness.

\begin{lem} \label{l:acyl}
Let $\mathcal{G}$ be a finite graph of groups. Suppose that for each $v \in V(\mathcal{G})$, the collection $( \iota_e(G_e) \mid e \in E(\mathcal{G}), i(e) = v )$ of subgroups of $G_v$ is uniformly almost malnormal. Then the action of $\pi_1(\mathcal{G})$ on the Bass--Serre tree $\mathcal{T_G}$ is acylindrical.
\end{lem}

\begin{proof}
Let $G = \pi_1(\mathcal{G})$ and $\mathcal{T} = \mathcal{T_G}$. By the assumption, and as the underlying graph of $\mathcal{G}$ is finite, there exists a constant $D < \infty$ such that for all $v \in V(\mathcal{G})$, all $e,f \in E(\mathcal{G})$ with $i(e) = v = i(f)$ and all $g \in G_v$, we have $|\iota_e(G_e) \cap \iota_f(G_f)^g| \leq D$ whenever $e \neq f$ or $g \notin \iota_e(G_e)$.

Now let $p,q \in \mathcal{T}$ be two points such that $d_{\mathcal{T}}(p,q) > 1$. We claim that $|\Stab_G(p) \cap \Stab_G(q)| \leq D$. Indeed, let $P \subseteq \mathcal{T}$ be the geodesic between $p$ and $q$. Then there exists a vertex $\widetilde{v} \in V(\mathcal{P})$ in the interior of $P$, and so there exist two distinct edges $\widetilde{e},\widetilde{f} \in E(\mathcal{T})$ with $i(\widetilde{e}) = \widetilde{v} = i(\widetilde{f})$ such that $P \cap \widetilde{e} \neq \varnothing$ and $P \cap \widetilde{f} \neq \varnothing$. In particular, any $g \in G$ stabilising $p$ and $q$ also stabilises $\widetilde{e}$ and $\widetilde{f}$.

Now let $v \in V(\mathcal{G})$ and $e,f \in E(\mathcal{G})$ be the images of $\widetilde{v}$ and $\widetilde{e},\widetilde{f}$ in $\mathcal{G}$, respectively. It then follows from the construction of the Bass--Serre tree $\mathcal{T}$ that $\Stab_G(\widetilde{v}) = G_v^h$ for some $h \in G$, and that $\Stab_G(\widetilde{e}) = \iota_e(G_e)^{h_eh} \leq G_v^h$ and $\Stab_g(\widetilde{f}) = \iota_f(G_f)^{h_fh} \leq G_v^h$ for some $h_e,h_f \in G_v$. Moreover, since $\widetilde{e} \neq \widetilde{f}$, we either have $e \neq f$ or $h_fh_e^{-1} \notin \iota_e(G_e)$. We thus have
\[
\Stab_G(p) \cap \Stab_G(q) \subseteq \Stab_G(\widetilde{e}) \cap \Stab_G(\widetilde{f}) = \left( \iota_e(G_e) \cap \iota_f(G_f)^{h_fh_e^{-1}} \right)^{h_eh}
\]
and so $|\Stab_G(p) \cap \Stab_G(q)| \leq D$, as claimed.

Now let $\varepsilon > 1$, and set $D_\varepsilon = 2\varepsilon+1$ and $N_\varepsilon = 2D(2\varepsilon+1)$. If $x,y \in \mathcal{T}$ are two points with $d_{\mathcal{T}}(x,y) \geq D_\varepsilon$, it follows from \cite[Lemma 4.2]{mo} that the set $\mathcal{N}_{x,y,\varepsilon} = \{ g \in G \mid d_{\mathcal{T}}(x,xg) \leq \varepsilon, d_{\mathcal{T}}(y,yg) \leq \varepsilon \}$ is contained in the union of at most $2(2\varepsilon+1)$ right cosets of $\Stab_G(p) \cap \Stab_G(q)$, where $p,q \in \mathcal{T}$ satisfy $d_{\mathcal{T}}(p,q) \geq \varepsilon > 1$. Thus $|\mathcal{N}_{x,y,\varepsilon}| \leq D \cdot 2(2\varepsilon+1) = N_\varepsilon$. Therefore, the action of $G$ on $\mathcal{T}$ is acylindrical, as required.
\end{proof}

The following consequence is almost contained in Lemma~\ref{l:acyl} and \cite[Theorem~1.4]{osin06def}. However, the latter result is only enough to show that the subgroups $\iota_e(G_e) \cap \iota_f(G_f)^g$ of $G_v$ are finite, but not to bound their order (here $g \in G_v$ and $e,f \in E(\mathcal{G})$ are such that $i(e) = i(f) = v$ and either $e \neq f$ or $g \notin \iota_e(G_e)$). We therefore give a complete proof of the following Corollary.

\begin{cor}
Let $\mathcal{G}$ be a finite graph of groups. Suppose that for each $v \in V(\mathcal{G})$, the group $G_v$ is finitely generated and relatively hyperbolic, and $( \iota_e(G_e) \mid e \in E(\mathcal{G}), i(e) = v )$ is a collection of distinct peripheral subgroups of $G_v$. Then the action of $\pi_1(\mathcal{G})$ on the Bass--Serre tree $\mathcal{T_G}$ is acylindrical.
\end{cor}

\begin{proof}
Let $v \in V(\mathcal{G})$. By Lemma~\ref{l:acyl}, it is enough to show that the collection $( \iota_e(G_e) \mid e \in E(\mathcal{G}), i(e) = v )$ of subgroups of $G_v$ is uniformly almost malnormal. To show this, we will use \cite[Theorem~3.23]{osin06def}.

Let $G = G_v$ and $\mathcal{H} = \bigsqcup_{e \in E(\mathcal{G}), i(e) = v} H_e$, where $H_e = \iota_e(G_e)$. We adopt the terminology and notation of \cite{osin06def}: in particular, let $X$ be a finite generating set for $G$ satisfying the assumptions in the beginning of \cite[\S3]{osin06def}, and denote by $\Gamma(G,X \sqcup \mathcal{H})$ the Cayley graph of $G$ with respect to $X \sqcup \mathcal{H}$. Let $\varepsilon = \varepsilon(1,2,0)$ be the constant given in \cite[Theorem~3.23]{osin06def}.

We claim that, given any $g \in G$ and $e,f \in E(\mathcal{G})$ with $i(e) = i(f) = v$, if either $e \neq f$ or $g \notin H_e$, then $|H_e \cap H_f^g| \leq D$, where $D$ is the number of elements in $G$ of word length $\leq \varepsilon$ with respect to $X$: since $X$ is finite, this will prove the result.

Thus, let $g \in G$ and $e,f \in E(\mathcal{G})$ with $i(e) = i(f) = v$, such that either $e \neq f$ or $g \notin H_e$, and consider $H_e \cap H_f^g$. Let $w \in (X \sqcup \mathcal{H})^*$ be a geodesic word representing $g$. Without loss of generality, suppose that no non-empty initial subword of $w$ represents an element of $H_f$: otherwise, we may replace $w$ with a shorter word and $g$ with another element. Similarly, assume that no non-empty terminal subword of $w$ represents an element of $H_e$: otherwise, we may replace $w$ with a shorter word, $g$ with another element, and the subgroup $H_e \cap H_f^g$ with its conjugate.

We now claim that $|h|_X \leq \varepsilon$ for all $h \in H_e \cap H_f^g$: this is enough to prove our original claim. Thus, let $h \in H_e \cap H_f^g$ be a non-trivial element, let $h' = ghg^{-1} \in H_f$, and let $P'Q'P^{-1}Q^{-1}$ be a cycle in $\Gamma(G,X \sqcup \mathcal{H})$, where $P$, $P'$, $Q$ and $Q'$ are geodesic paths labelled by the $H_e$-letter $h$, the $H_f$-letter $h'$, the word $w$ and the word $w$, respectively. It follows from the assumptions on $w$ that both $Q'$ and $Q^{-1}$ are components of this cycle and that they are not connected to any component of $P'$ or of $P^{-1}$. Since $P$ and $P'$ are geodesics and $|Q| = |Q'| = 1$, it then follows that $QP$ and $P'Q'$ are $0$-similar $(1,2)$-quasi-geodesics without backtracking. We have already seen that $Q$ is a component of $QP$ not connected to any component of $P'$; moreover, since either $e \neq f$ or $g \notin H_e$, it follows that $Q$ is not connected to $Q'$. It follows from \cite[Theorem~3.23(2)]{osin06def} that $|h|_X \leq \varepsilon$, as claimed.
\end{proof}

\subsection{Equationally Noetherian vertex groups} \label{ss:combEN}

Here we prove the implication ($\Leftarrow$) in Theorem~\ref{t:comb}\ref{i:tcomb-EN}.

Acylindricity of the action of $\pi_1(\mathcal{G})$ on $\mathcal{T_G}$ allows us to apply the following result of \cite{gh}. Here, let a group $G$ act non-elementarily acylindrically on a fixed hyperbolic space $X$, let $F$ be a group with a finite generating set $S$, and let $\omega$ be a non-principal ultrafilter. We then say that a sequence of homomorphisms $(\varphi_j: F \to G)_{j \in \mathbb{N}}$ is \emph{non-divergent} if
\[
\lim_\omega \inf_{x \in X} \max_{s \in S} d_X(x,x\cdot\varphi_j(s)) < \infty.
\]
In the view of Theorem~\ref{t:qae}, the following result then says that in order to show that an acylindrically hyperbolic group is equationally Noetherian, only considering non-divergent sequences of homomorphisms is enough.
\begin{thm}[D.~Groves and M.~Hull {\cite[Theorem B]{gh}}] \label{t:gh}
Let $G$ be an acylindrically hyperbolic group. Suppose that for any finitely generated group $F$ and any non-divergent sequence of homomorphisms $(\varphi_j: F \to G)_{j \in \mathbb{N}}$, we have $K_{\{1\}}^{(\varphi_j)} \subseteq \ker(\varphi_j)$ $\omega$-almost surely. Then $G$ is equationally Noetherian.
\end{thm}

The following result is the key part in our proof of Theorem~\ref{t:comb}\ref{i:tcomb-EN}. It shows that, under the assumptions of Theorem~\ref{t:comb}\ref{i:tcomb-EN}, the group $\pi_1(\mathcal{G})$ and subset $A = \{1\}$ satisfy condition \ref{i:tqae-as} of Theorem~\ref{t:qae} for some specific sequences of homomorphisms.

\begin{prop} \label{p:comb}
Let $\mathcal{G}$ be a finite graph of groups satisfying the assumptions \ref{i:tcomb-ext} and \ref{i:tcomb-fin1v} in Theorem~\ref{t:comb}, and suppose that each group $G_v$ is equationally Noetherian and each subgroup $\iota_e(G_e) \leq G_{i(e)}$ is quasi-algebraic. Let $\{ F(v) \}_{v \in V(\mathcal{G})}$ be a collection of finitely generated free groups, let $F_E$ be a free group with free basis $\{ X_e \mid e \in E(\mathcal{G}) \}$, and let $F = (\ast_{v \in V(\mathcal{G})} F(v) ) * F_E$. Let $\varphi_j: F \to \pi_1(\mathcal{G})$ be a sequence of homomorphisms such that, for all $j$, $\varphi_j(F(v)) \subseteq G_v$ for each $v \in V(\mathcal{G})$ and $\varphi_j(X_e) = e$ for each $e \in E(\mathcal{G})$. Then, for any non-principal ultrafilter $\omega$, $K^{(\varphi_j)}_{\{1\}} \subseteq \ker(\varphi_j)$ $\omega$-almost surely.
\end{prop}

\begin{proof}
The idea of the proof is to extend each map $\varphi_j: F \to \pi_1(\mathcal{G})$ to a map $\overline\varphi_j: \overline{F} \to \pi_1(\mathcal{G})$. Here the free group $\overline{F}$ is formed by adding new elements to a free basis of $F$, using assumption \ref{i:tcomb-ext} of Theorem~\ref{t:comb}, so that if we have a basis element mapping (under $\overline\varphi_j$) to $g \in G_{i(e)}$ then we also have one mapping to $\overline\phi_e(g) \in G_{i(\overline{e})}$. By assumption \ref{i:tcomb-fin1v} and Lemma~\ref{l:fin1v2v}, we can do this in such a way that $\overline{F}$ is still finitely generated. We then use the assumptions on equational Noetherianity of the $G_v$ and quasi-algebraicity of the $\iota_e(G_e)$ to describe the normal subgroup $K^{(\overline\varphi_j)}_{\{1\}} \unlhd \overline{F}$.

Given an edge $e \in E(\mathcal{G})$, let $\overline\phi_e: G_{i(e)} \to G_{i(\overline{e})}$ be the homomorphism given by assumption \ref{i:tcomb-ext} of Theorem~\ref{t:comb}. It then follows from assumption \ref{i:tcomb-fin1v} and Lemma~\ref{l:fin1v2v} that for each $v,w \in V(\mathcal{G})$, there exists a finite set of homomorphisms $\eta_{v,w,1},\ldots,\eta_{v,w,d(v,w)}: G_v \to G_w$ such that for every path $e_1 \cdots e_n$ in $\mathcal{G}$ from $v$ to $w$, we have $\overline\phi_{e_n} \circ \cdots \circ \overline\phi_{e_1} = \eta_{v,w,\delta}$ for some $\delta \in \{ 1,\ldots,d(v,w) \}$. Without loss of generality, assume that all the $\eta_{v,w,\delta}$ are distinct and arise this way for some path, and that $\eta_{v,v,1}: G_v \to G_v$ is the identity for each $v \in V(\mathcal{G})$. 

Now for each $v,w \in V(\mathcal{G})$ and each $\delta \in \{1,\ldots,d(w,v)\}$, let $F_{w,\delta}(v)$ be a free group isomorphic to $F(w)$, with an isomorphism $\zeta_{w,v,\delta}: F(w) \to F_{w,\delta}(v)$. For each $v \in V(\mathcal{G})$, we form a free group $\overline{F}(v) = \displaystyle\mathop\ast_{\substack{w \in V(\mathcal{G}) \\ 1 \leq \delta \leq d(w,v)}} F_{w,\delta}(v)$, and define $\overline{F} = \left( \displaystyle\mathop\ast_{v \in V(\mathcal{G})} \overline{F}(v) \right) * F_E$. We identify the subgroup $F_{v,1}(v) \leq \overline{F}(v)$ with $F(v)$ via the isomorphism $\zeta_{v,v,1}$: this gives an injective group homomorphism $\alpha: F \to \overline{F}$, defined as $\zeta_{v,v,1}$ on $F(v)$ and as identity on $F_E$.

We now define homomorphisms $\overline\varphi_j: \overline{F} \to \pi_1(\mathcal{G})$ by setting $\overline\varphi_j(X_e) = e$ for every $e \in E(\mathcal{G})$, and $\overline\varphi_j(g) = \eta_{v,w,\delta} (\varphi_j (\zeta_{v,w,\delta}^{-1}(g))) \in G_w$ for every $v,w \in V(\mathcal{G})$, every $\delta \in \{1,\ldots,d(v,w)\}$ and every $g \in F_{v,\delta}(w)$. It follows from the construction that $\varphi_j = \overline\varphi_j \circ \alpha$ and that $\overline\varphi(F(v)) \subseteq G_v$ for each $j$ and each $v \in V(\mathcal{G})$. It is therefore enough to show that $K^{(\overline\varphi_j)}_{\{1\}} \subseteq \ker(\overline\varphi_j)$ $\omega$-almost surely, since in that case we also have $K^{(\varphi_j)}_{\{1\}} = \alpha^{-1}\left(K^{(\overline\varphi_j)}_{\{1\}}\right) \subseteq \alpha^{-1}\left(\ker(\overline\varphi_j)\right) = \ker(\varphi_j)$ $\omega$-almost surely.

It follows from the construction that for each vertex $v \in V(\mathcal{G})$, each edge $e \in E(\mathcal{G})$ and each $\delta \in \{1,\ldots,d(v,i(e))\}$, there exists a unique integer $\overline\delta = \overline\delta(v,e,\delta) \in \{1,\ldots,d(v,i(\overline{e}))\}$ such that $\overline\phi_e \circ \eta_{v,i(e),\delta} = \eta_{v,i(\overline{e}),\overline\delta}$. This allows us to define a homomorphism $\Phi_e: \overline{F}(i(e)) \to \overline{F}(i(\overline{e}))$ by setting the restriction of $\Phi_e$ to $F_{v,\delta}(i(e))$ to be $\zeta_{v,i(\overline{e}),\overline\delta} \circ \zeta_{v,i(e),\delta}^{-1} : F_{v,\delta}(i(e)) \to F_{v,\overline\delta}(i(\overline{e}))$, where $\overline\delta = \overline\delta(v,e,\delta)$ is as in the previous sentence. This implies that we have, for each $j \in \mathbb{N}$, $v \in V(\mathcal{G})$, $e \in E(\mathcal{G})$, $\delta \in \{1,\ldots,d(v,i(e))\}$ and $g \in F_{v,\delta}(i(e))$,
\begin{align*}
(\overline\phi_e \circ \overline\varphi_j)(g) &= (\overline\phi_e \circ \eta_{v,i(e),\delta} \circ \varphi_j \circ \zeta_{v,i(e),\delta}^{-1})(g) = (\eta_{v,i(\overline{e}),\overline\delta} \circ \varphi_j \circ \zeta_{v,i(e),\delta}^{-1})(g) \\ &= \left((\eta_{v,i(\overline{e}),\overline\delta} \circ \varphi_j \circ \zeta_{v,i(\overline{e}),\overline\delta}^{-1}) \circ (\zeta_{v,i(\overline{e}),\overline\delta} \circ \zeta_{v,i(e),\delta}^{-1})\right)(g) = (\overline\varphi_j \circ \Phi_e)(g),
\end{align*}
and so $\overline\phi_e \circ \overline\varphi_j = \overline\varphi_j \circ \Phi_e$ 
for all $j \in \mathbb{N}$ and all $e \in E(\mathcal{G})$; here we have been abusing notation slightly by modifying the domains and codomains of maps in such a way that we could compose them.

Thus -- to recap -- we have a collection $\{ \overline{F}(v) \mid v \in V(\mathcal{G}) \}$ of finitely generated free groups and two collections of homomorphisms, $\{ \Phi_e: \overline{F}(i(e)) \to \overline{F}(i(\overline{e})) \mid e \in E(\mathcal{G}) \}$ and $\{ \overline\varphi_j: \overline{F} \to \pi_1(\mathcal{G}) \mid j \in \mathbb{N} \}$ (where $\overline{F} = (\ast_{v \in V(\mathcal{G})} \overline{F}(v) ) * F_E$), such that $\overline\varphi_j(\overline{F}(v)) \subseteq G_v$, $\overline\varphi_j(X_e) = e$ and $\overline\varphi_j(\Phi_e(g)) = \overline\phi_e(\overline\varphi_j(g))$ for all $j \in \mathbb{N}$, $v \in V(\mathcal{G})$, $e \in E(\mathcal{G})$ and $g \in \overline{F}(i(e))$. It follows that for all $e \in E(\mathcal{G})$ and all $g \in \overline{F}(i(e))$ we have $\overline\varphi_j(X_e^{-1} g X_e) = \overline\varphi_j(\Phi_e(g))$ if and only if $e^{-1} \overline\varphi_j(g) e = \overline\phi_e(\overline\varphi_j(g))$, which by Theorem~\ref{t:ftbst} happens if and only if $\overline\varphi_j(g) \in \iota_e(G_e)$.

Now let $\mathcal{K} = \mathcal{K}_E \cup \mathcal{K}_V \cup \mathcal{K}_0$, where
\[
\mathcal{K}_E = \bigcup_{e \in E(\mathcal{G})} \left\{ \Phi_e(g)^{-1} X_e^{-1} g X_e \:\middle|\: g \in K_{\iota_e(G_e)}^{\left(\overline\varphi_j|_{\overline{F}(i(e))}\right)} \right\},
\qquad
\mathcal{K}_V = \bigcup_{v \in V(\mathcal{G})} K_{\{1\}}^{\left(\overline\varphi_j|_{\overline{F}(v)}\right)},
\]
and
\[ \mathcal{K}_0 = \{ X_e \mid e \in E(\mathcal{G}), e=1 \text{ in } \pi_1(\mathcal{G}) \} \cup \{ X_eX_{\overline{e}} \mid e \in E(\mathcal{G}) \}.
\]
We claim that $K^{(\overline\varphi_j)}_{\{1\}} = \langle\!\langle \mathcal{K} \rangle\!\rangle$, the normal closure of $\mathcal{K}$ in $\overline{F}$. Indeed, it is clear from the construction that $\mathcal{K} \subseteq K^{(\overline\varphi_j)}_{\{1\}}$. Conversely, let $g \in K^{(\overline\varphi_j)}_{\{1\}}$, so that $\overline\varphi_j(g) = 1$ $\omega$-almost surely. After inserting letters of $\mathcal{K}_0$ to $g$ (that is, multiplying $g$ by an element of $\langle\!\langle \mathcal{K}_0 \rangle\!\rangle$) if necessary, we may assume that $g$ has the form $g = g_0 X_{e_1} g_1 X_{e_2} \cdots X_{e_n} g_n$ where $g_\ell \in \overline{F}(v_\ell)$ are such that $i(e_\ell) = v_{\ell-1}$ and $i(\overline{e_\ell}) = v_\ell$ for $1 \leq \ell \leq n$, and $v_0 = v_n$. We argue by induction on $n$ that $g \in \langle\!\langle \mathcal{K} \rangle\!\rangle$. Indeed, if $n = 0$ then we have $g \in \mathcal{K}_V \subseteq \langle\!\langle \mathcal{K} \rangle\!\rangle$. Otherwise, it follows from Theorem~\ref{t:ftbst} and from finite additivity of $\omega$ that for some $\ell \in \{1,\ldots,n-1\}$ we $\omega$-almost surely have $e_{\ell+1} = \overline{e_\ell}$ and $\overline\varphi_j(g_\ell) \in \iota_{\overline{e_\ell}}(G_{e_\ell})$. It thus follows that $\Phi_{\overline{e_\ell}}(g_\ell)^{-1} X_{\overline{e_\ell}}^{-1} g_\ell X_{\overline{e_\ell}} \in \mathcal{K}_E$, and so $g$ is equal to $g' = g_0 X_{e_1} g_1 \cdots X_{e_{\ell-1}} [ g_{\ell-1} \Phi_{\overline{e_\ell}}(g_\ell) g_{\ell+1} ] X_{e_{\ell+2}} \cdots X_{e_n} g_n$ modulo $\langle\!\langle \mathcal{K}_E \cup \mathcal{K}_0 \rangle\!\rangle$, where $g_{\ell-1} \Phi_{\overline{e_\ell}}(g_\ell) g_{\ell+1} \in \overline{F}(v_{\ell-1}) = \overline{F}(v_{\ell+1})$. It then follows, by the inductive hypothesis, that $g' \in \langle\!\langle \mathcal{K} \rangle\!\rangle$, and so $g \in \langle\!\langle \mathcal{K} \rangle\!\rangle$ as well. Hence $K^{(\overline\varphi_j)}_{\{1\}} = \langle\!\langle \mathcal{K} \rangle\!\rangle$, as claimed.

Finally, note that we $\omega$-almost surely have $K_{\{1\}}^{\left(\overline\varphi_j|_{\overline{F}(v)}\right)} \subseteq \ker\left(\overline\varphi_j|_{\overline{F}(v)}\right)$ for all $v \in V(\mathcal{G})$ (as $G_v$ is equationally Noetherian) and $K_{\iota_e(G_e)}^{\left(\overline\varphi_j|_{\overline{F}(i(e))}\right)} \subseteq \left(\overline\varphi_j|_{\overline{F}(i(e))}\right)^{-1}(\iota_e(G_e))$ for all $e \in E(\mathcal{G})$ (as $\iota_e(G_e) \leq G_{i(e)}$ is quasi-algebraic). This implies that $\mathcal{K}_V \subseteq \ker(\overline\varphi_j)$ and $\mathcal{K}_E \subseteq \ker(\overline\varphi_j)$ $\omega$-almost surely. As we also have $\mathcal{K}_0 \subseteq \ker(\overline\varphi_j)$ for all $j$ by construction, and as $\ker(\overline\varphi_j)$ is a normal subgroup of $\overline{F}$ for each $j$, it follows that $K^{(\overline\varphi_j)}_{\{1\}} = \langle\!\langle \mathcal{K} \rangle\!\rangle \subseteq \ker(\overline\varphi_j)$ $\omega$-almost surely. This concludes the proof.
\end{proof}

We now combine Theorem~\ref{t:gh} and Proposition~\ref{p:comb} to prove the `if' part of Theorem~\ref{t:comb}\ref{i:tcomb-EN}. To do this, we adapt the first half of the proof of Theorem \cite[Theorem D]{gh} to our scenario.

\begin{prop} \label{p:comb-EN-<=}
Let $\mathcal{G}$ be a finite graph of groups. Suppose that
\begin{enumerate}[label={\textup{(\roman*)}}]
\item \label{i:pcomb-en-ext} for each $e \in E(\mathcal{G})$, the isomorphism $\phi_e = \iota_{\overline{e}} \circ \iota_e^{-1}: \iota_e(G_e) \to \iota_{\overline{e}}(G_e)$ extends to a homomorphism $\overline\phi_e: G_{i(e)} \to G_{i(\overline{e})}$;
\item \label{i:pcomb-en-fin1v} for each $v \in V(\mathcal{G})$, there are only finitely many endomorphisms of the form $\overline\phi_{e_n} \circ \cdots \circ \overline\phi_{e_1}: G_v \to G_v$, where $e_1 \cdots e_n$ ranges over all closed paths in $\mathcal{G}$ starting and ending at $v$;
\item \label{i:pcomb-en-en} for each $v \in V(\mathcal{G})$, the group $G_v$ is equationally Noetherian;
\item \label{i:pcomb-en-qa} for each $e \in E(\mathcal{G})$, the subgroup $\iota_e(G_e) \leq G_{i(e)}$ is quasi-algebraic; and
\item \label{i:pcomb-en-acyl} the action of $\pi_1(\mathcal{G})$ on $\mathcal{T_G}$ is acylindrical.
\end{enumerate}
Then $\pi_1(\mathcal{G})$ is equationally Noetherian.
\end{prop}

\begin{proof}
Let $G = \pi_1(\mathcal{G})$ and let $\mathcal{T} = \mathcal{T_G}$ be the corresponding Bass--Serre tree, so that the action of $G$ on $\mathcal{T}$ is acylindrical by \ref{i:pcomb-en-acyl}. We may assume that $G$ is not virtually cyclic (otherwise $G$ is equationally Noetherian since it is linear, for instance) and that $\mathcal{T}$ is unbounded (otherwise $G = G_v$ for some $v \in V(\mathcal{G})$, and so it is equationally Noetherian by assumption \ref{i:pcomb-en-en}). As the action of $G$ on $\mathcal{T}$ is cobounded, these two assumptions imply that it is also non-elementary; in particular, $G$ is acylindrically hyperbolic.

As $\mathcal{T}$ is a simplicial tree, so $0$-hyperbolic, we may use Theorem~\ref{t:gh} to show that $G$ is equationally Noetherian. Thus, let $F$ be a group with a finite generating set $S$, and let $(\varphi_j: F \to G)_{j \in \mathbb{N}}$ be a non-divergent sequence of homomorphisms.

We fix an edge $e \in E(\mathcal{G})$, and let $x_0 \in \mathcal{T}$ be the midpoint of the edge $\widetilde{e} \in E(\mathcal{T})$ which maps to $e$ under the quotient map $\mathcal{T} \to \mathcal{T}/G$ and whose stabiliser is $G_e$. Since the action of $G$ on $\mathcal{T}$ is cobounded, we may assume that, after conjugating each $\varphi_j$ if necessary, we have $\lim_\omega \max_{s \in S} d_\mathcal{T}(x_0,x_0 \cdot \varphi_j(s)) = n' < \infty$. Since $\omega$ is finitely additive and $S$ is finite, it follows that there exist $\{ n_s' \mid s \in S \} \subseteq \{0,\ldots,n'\}$ such that $d_\mathcal{T}(x_0,x_0 \cdot \varphi_j(s)) = n_s'$ for every $s \in S$ $\omega$-almost surely. Again, since $\omega$ is finitely additive and $S$, $V(\mathcal{G})$ and $E(\mathcal{G})$ are all finite, this implies that for each $s \in S$ there exist $n_s \in \mathbb{N}$, vertices $v_{s,0},\ldots,v_{s,n_s} \in V(\mathcal{G})$ and edges $e_{s,1},\ldots,e_{s,n_s} \in E(\mathcal{G})$ such that we $\omega$-almost surely have
\[
\varphi_j(s) = g_{s,0,j} e_{s,1} g_{s,1,j} \cdots e_{s,n_s} g_{s,n_s,j}
\]
with $g_{s,\ell,j} \in G_{v_{s,\ell}}$ for each $\ell$.

Now for each $v \in V(\mathcal{G})$, let $\widehat{S}(v) = \{ (s,\ell) \mid v_{s,\ell} = v \}$, and let $\widehat{F}(v)$ be a free group with free basis $\{ Y_{(s,\ell)} \mid (s,\ell) \in \widehat{S}(v) \}$. Let $\widehat{F} = \left( \ast_{v \in V(\mathcal{G})} \widehat{F}(v) \right) * F_E$, where $F_E$ is a free group with free basis $\{ X_e \mid e \in E(\mathcal{G}) \}$. We define a sequence of homomorphisms $(\widehat\varphi_j: \widehat{F} \to G)_{j \in \mathbb{N}}$ by setting $\widehat\varphi_j(Y_{(s,\ell)}) = g_{s,\ell,j}$ and $\widehat\varphi_j(X_e) = e$, so that $\widehat\varphi_j$ is well-defined $\omega$-almost surely. It then follows by Proposition~\ref{p:comb} that $K_{\{1\}}^{(\widehat\varphi_j)} \subseteq \ker(\widehat\varphi_j)$ $\omega$-almost surely.

Now define a map $\alpha: S \to \widehat{F}$ by setting $\alpha(s) = Y_{(s,0)} X_{e_{s,1}} Y_{(s,1)} \cdots X_{e_{s,n_s}} Y_{(s,n_s)}$, and extend it to a map $\alpha: F_S \to \widehat{F}$, where $F_S$ is a free group with free basis $S$. Let $\pi: F_S \to F$ be the canonical surjection. It then follows from the construction that $\varphi_j \circ \pi = \widehat\varphi_j \circ \alpha$ $\omega$-almost surely. Now as $K_{\{1\}}^{(\widehat\varphi_j)} \subseteq \ker(\widehat\varphi_j)$ $\omega$-almost surely, we also have $K_{\{1\}}^{(\varphi_j \circ \pi)} = \alpha^{-1} \left( K_{\{1\}}^{(\widehat\varphi_j)} \right) \subseteq \alpha^{-1} \left( \ker(\widehat\varphi_j) \right) = \ker(\varphi_j \circ \pi)$ $\omega$-almost surely. As $\pi$ is surjective, this implies that $K_{\{1\}}^{(\varphi_j)} \subseteq \ker(\varphi_j)$ $\omega$-almost surely. Thus, by Theorem~\ref{t:gh}, $G$ is equationally Noetherian, as required.
\end{proof}

\subsection{Necessary conditions and remarks} \label{ss:comb-rem}

Finally, we complete the proof of Theorem~\ref{t:comb} -- that is, prove the implications ($\Rightarrow$) -- and make a couple of remarks.

\begin{lem} \label{l:comb-=>}
Let $\mathcal{G}$ be a finite graph of groups. Suppose that for each $e \in E(\mathcal{G})$, the isomorphism $\phi_e = \iota_{\overline{e}} \circ \iota_e^{-1}: \iota_e(G_e) \to \iota_{\overline{e}}(G_e)$ extends to a homomorphism $\overline\phi_e: G_{i(e)} \to G_{i(\overline{e})}$. Then the following hold.
\begin{enumerate}[label=\textup{(\alph*)}]
\item \label{i:lcomb-RF} If $\pi_1(\mathcal{G})$ is residually finite, then $\iota_e(G_e)$ is separable in $G_{i(e)}$ for all $e \in E(\mathcal{G})$.
\item \label{i:lcomb-EN} If $\pi_1(\mathcal{G})$ is equationally Noetherian, then $\iota_e(G_e)$ is quasi-algebraic in $G_{i(e)}$ for all $e \in E(\mathcal{G})$.
\end{enumerate}
\end{lem}

\begin{proof}
We first claim that given any $e \in E(\mathcal{G})$ and any $h \in G_{i(e)}$, we have $h \in \iota_e(G_e)$ if and only if $e \overline\phi_e(h)^{-1} e^{-1} h = 1$. Indeed, the implication ($\Rightarrow$) is clear. Conversely, suppose that $h \notin \iota_e(G_e)$. If $\overline\phi_e(h) \notin \iota_{\overline{e}}(G_e)$ then it follows from Theorem~\ref{t:ftbst} that $e \overline\phi_e(h)^{-1} e^{-1} h \neq 1$; otherwise, we have $e\overline\phi_e(h)e^{-1} \in \iota_e(G_e)$ but $h \notin \iota_e(G_e)$, and so again $\left( e \overline\phi_e(h) e^{-1} \right)^{-1} h$ is non-trivial. This proves the claim.

\begin{enumerate}[label=(\alph*)]

\item Suppose for contradiction that $\iota_e(G_e)$ is not separable in $G_{i(e)}$ for some $e \in E(\mathcal{G})$, and let $h \in G_{i(e)} \setminus \iota_e(G_e)$ be an element such that $hK \in \iota_e(G_e)K / K$ for any finite-index normal subgroup $K \unlhd G_{i(e)}$. Consider the element $g = e\overline\phi_e(h)^{-1}e^{-1}h \in \pi_1(\mathcal{G})$: it is non-trivial by the claim above.

Let $N \unlhd \pi_1(\mathcal{G})$ be a normal finite-index subgroup. Then the groups $K_{i(e)} := N \cap G_{i(e)}$ and $K_{i(\overline{e})} := N \cap G_{i(\overline{e})}$ are finite-index normal subgroups of $G_{i(e)}$ and $G_{i(\overline{e})}$, respectively, and so $K := K_{i(e)} \cap (\overline\phi_e)^{-1}(K_{i(\overline{e})})$ is a finite-index normal subgroup of $G_{i(e)}$. We thus have $hK = h'K$ for some $h' \in \iota_e(G_e)$; therefore, $e^{-1}h'e = \overline\phi_e(h')$, and so $gK = e\overline\phi_e(h)^{-1}e^{-1}h'K = e\overline\phi_e(h^{-1}h')e^{-1}K$. But as $h^{-1}h' \in K \subseteq (\overline\phi_e)^{-1}(K_{i(\overline{e})})$, it follows that $\overline\phi_e(h^{-1}h') \in K_{i(\overline{e})} \subseteq N$. Hence, as $K \subseteq N$ and as $N$ is normal, we have $gN = e\overline\phi_e(h^{-1}h')e^{-1}N = N$. This contradicts the fact that $\pi_1(\mathcal{G})$ is residually finite; thus $\iota_e(G_e)$ must be separable in $G_{i(e)}$ for all $e \in E(\mathcal{G})$.

\item Suppose that $\pi_1(\mathcal{G})$ is equationally Noetherian, let $e \in E(\mathcal{G})$, and let $S \subseteq F_n$ be a system of equations. Consider the system of equations $\overline{S} = \{ \Xi(s) \mid s \in S \} \subseteq F_{2n+1}(X_1,\ldots,X_{2n+1})$, where
\[
\Xi(s) = X_{2n+1} s(X_1,\ldots,X_n)^{-1} X_{2n+1}^{-1} s(X_{n+1},\ldots,X_{2n})
\]
for any $s \in F_n$. It then follows from the claim above that for any $g_1,\ldots,g_n \in G_{i(e)}$, we have $(g_1,\ldots,g_n) \in V_{G_{i(e)},\iota_e(G_e)}(S)$ if and only if $(\overline\phi_e(g_1),\ldots,\overline\phi_e(g_n),g_1,\ldots,g_n,e) \in V_{\pi_1(\mathcal{G})}(\overline{S})$. But as $\pi_1(\mathcal{G})$ is equationally Noetherian, there exists a finite subset $\overline{S}_0 = \{ \Xi(s) \mid s \in S_0 \}$ of $\overline{S}$ such that $V_{\pi_1(\mathcal{G})}(\overline{S}) = V_{\pi_1(\mathcal{G})}(\overline{S}_0)$, where $S_0$ is some finite subset of $S$. It follows that $V_{G_{i(e)},\iota_e(G_e)}(S) = V_{G_{i(e)},\iota_e(G_e)}(S_0)$; hence $\iota_e(G_e)$ is quasi-algebraic in $G_{i(e)}$. \qedhere

\end{enumerate}
\end{proof}

\begin{proof}[Proof of Theorem~\ref{t:comb}]
The implications ($\Leftarrow$) follow from Proposition~\ref{p:comb-RF-<=} for \ref{i:tcomb-RF}, and from Proposition~\ref{p:comb-EN-<=} for \ref{i:tcomb-EN}.

The implications ($\Rightarrow$) follow from Lemma~\ref{l:comb-=>} and from the fact that subgroups of residually finite (respectively equationally Noetherian) groups are residually finite (respectively equationally Noetherian).
\end{proof}

\begin{rmk}
In both parts \ref{i:tcomb-RF} and \ref{i:tcomb-EN} of Theorem~\ref{t:comb}, we have actually proved the implication ($\Rightarrow$) in more generality than stated. In particular, neither of the proofs for these implications uses the assumption \ref{i:tcomb-fin1v} in Theorem~\ref{t:comb} -- and the acylindricity assumption in part \ref{i:tcomb-EN} is not used either.
\end{rmk}

\begin{rmk}
In general, the assumption on the existence of maps $\overline\phi_e$ defined in Lemma~\ref{l:comb-=>} cannot be dropped. For instance, consider the group $G = \langle a,b,t \mid a^t = a^2 \rangle$. It can be expressed as an amalgamated free product of the subgroups $H = \langle a,t \rangle \cong BS(1,2)$ and $K = \langle a,b \rangle \cong F_2$ with amalgamated subgroup $L = \langle a \rangle \cong \mathbb{Z}$, as well as a free product of $H$ and $\langle b \rangle \cong \mathbb{Z}$. The latter expression implies that $G$ is both residually finite and equationally Noetherian (as it is a free product of two finitely generated linear groups, for example).

However, it is easy to check that $tat^{-1} \notin L$ but $tat^{-1}N \in LN/N$ for every finite-index normal subgroup $N \unlhd H$, showing that $L$ is not separable in $H$. Likewise, if we set $s_n = X_1^n X_2 X_1^{-n} \in F_2(X_1,X_2)$ for $n \in \mathbb{N}$, we have $s_n(t,a^{2^m}) \in L$ if and only if $m \geq n$, and so for any finite subset $S_0 \subset S := \{ s_n \mid n \in \mathbb{N} \}$ there exists $m \in \mathbb{N}$ such that $(t,a^{2^m}) \in V_{H,L}(S_0) \setminus V_{H,L}(S)$; thus $L$ is not quasi-algebraic in $H$.

It follows that, with respect to the splitting $G \cong H \ast_L K$, neither statement~\ref{i:lcomb-RF} nor statement~\ref{i:lcomb-EN} in Lemma~\ref{l:comb-=>} is true. Note also that the action of $G$ on the Bass--Serre tree $\mathcal{T}$ corresponding to this splitting is acylindrical: indeed, using the splitting $G \cong H \ast \langle b \rangle$, Theorem~\ref{t:ftbst} implies that $H \cap H^g = \{1\}$ for every $g \in G \setminus H$, which implies acylindricity (of the action on $\mathcal{T}$) by \cite[Lemma~4.2]{mo}.
\end{rmk}

\bibliographystyle{amsalpha}
\bibliography{../../../all}

\end{document}